\documentclass{conm-p-l}

\usepackage{graphicx}

\newtheorem{theorem}{Theorem}
\newtheorem{lemma}[theorem]{Lemma}
\newtheorem{prop}[theorem]{Proposition}

\newtheorem{conj}[theorem]{Conjecture}

\theoremstyle{definition}

\theoremstyle{remark}

\numberwithin{equation}{section}

\newcommand{\bibtitle}[1]{\emph{#1}}
\newcommand{\On}[1]{\mathcal{O}_{#1}}
\newcommand{\dfn}[1]{\textbf{#1}}
\newcommand{\ty}{\nabla\mathrm{Y}}
\newcommand{\yt}{\mathrm{Y}\nabla}
\newcommand{\yty}{\mathrm{Y}\nabla\mathrm{Y}}

\begin{document}

\title{The $K_{n+5}$ and $K_{3^2,1^n}$ families are obstructions to $n$-apex.}



\author{Thomas W.\ Mattman}
\address{Department of Mathematics and Statistics,
California State University, Chico,
Chico, CA 95929-0525}
\email{TMattman@CSUChico.edu}

\author{Mike Pierce}
\address{Department of Mathematics,
UC, Riverside,
900 University Avenue,
Riverside, CA 92521}
\email{mpierce@math.ucr.edu}


\date{}

\begin{abstract}
We give evidence in support of a conjecture that the $\yty$ families of $K_{n+5}$ and $K_{3^2,1^n}$ are minor minimal obstructions for the $n$-apex property
\end{abstract}

\maketitle

\section*{Introduction}
A graph is \dfn{$n$-apex} if deleting $n$ or fewer vertices results in a 
planar graph. Kuratowski~\cite{K} showed the $0$-apex or \dfn{planar} graphs are characterized by the obstruction set 
$\On{0} = \{K_5, K_{3,3}\}$. Using the formulation of Wagner~\cite{W}, a graph is planar if and only if it has neither $K_5$ nor $K_{3,3}$
as a minor. Recall that $H$ is a \dfn{minor} of the graph $G$ if $H$ can 
be obtained by contracting edges in a subgraph of $G$. As the $n$-apex property
is minor closed, it follows from the Graph Minor Theorem of Robertson and Seymour~\cite{RS} that there is a corresponding finite obstruction set $\On{n}$.
For each $n$, $\On{n}$ is the set of graphs that are not $n$-apex even though 
every proper minor is. We say that the elements of $\On{n}$ are \dfn{minor minimal
not $n$-apex} or MMNnA.

Determining this obstruction set is likely to be quite difficult. It's known that, already,
 $|\On{1}| > 150$~\cite{P}, and we expect that 
 $|\On{n}|$ grows quickly in size with $n$.
In the current paper we report on computer searches that show $|\On{2}| \geq 82$,
$|\On{3}| \geq 601$,  $|\On{4}| \geq 520$, and $|\On{5}| \geq 608$.

\begin{figure}[ht]
 \centering
 \includegraphics[scale=0.8]{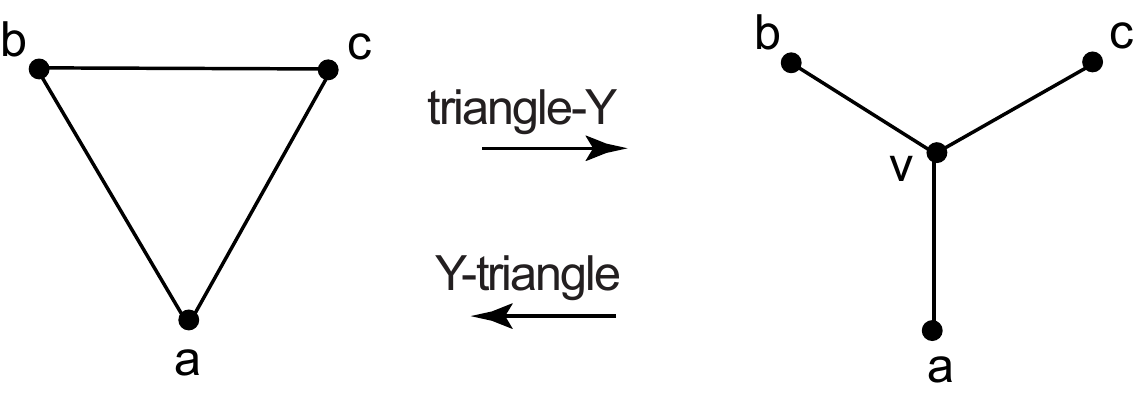}
 \caption{$\ty$ and $\yt$ moves.
 \label{figtymove}
 }
\end{figure}

We obtain these bounds on the size of $\On{n}$ using the $\yty$ families 
of $K_{n+5}$, the complete graph on $n+5$ vertices, and $K_{3^2,1^n}$, 
the complete multi-partite graph with two parts of three vertices and a further $n$
parts that each have a single vertex. As in Figure~\ref{figtymove},
we perform a $\ty$ or \dfn{triangle-Y move} on a graph by deleting the edges of a triangle $abc$ and then adding a vertex $v$ and edges $av$, $bv$, and $cv$. The reverse operation is a $\yt$ or \dfn{Y-triangle move}. The set of graphs
obtained from a graph $G$ through a sequence of zero or more $\ty$ or $\yt$ 
moves is \dfn{$G$'s ($\yty$) family}. 

In this paper we provide evidence in support of the following conjecture.
\begin{conj}
\label{conmain}%
For each $n > 0$, the $K_{n+5}$ and $K_{3^2,1^n}$ families are in $\On{n}$.
\end{conj}
When $n=0$, there are 49 graphs in the $K_5$ family and 10 in $K_{3,3}$'s. However, $\On{0} = \{K_5, K_{3,3} \}$ as proved by Weber~\cite{W} and the analogue
of the conjecture does not hold in this case.
On the other hand, a proof for $n \leq 2$ follows easily from earlier work.
\begin{prop} 
\label{prong2}%
The conjecture holds for $n\leq 2$.
\end{prop}
\begin{proof}
For $n=1$, the 
$K_6$ and $K_{3,3,1}$ families coincide. This collection of seven graphs is 
known as the Petersen family since it also includes the Petersen graph.
Barsotti and Mattman~\cite{BM} observed that this family is in $\On{1}$
on the way to proving that the $K_7$ family is in $\On{2}$. In fact,
they showed that the $K_7$ family is exactly the set of graphs in $\On{2}$ 
of size 21 or less.

It remains to show that $K_{3^2,1^2}$ is in $\On{2}$. In \cite{GMN}, 
Goldberg, Mattman, and Naimi
show that the 58 graphs in this family are all minor minimal intrinsically knotted
or MMIK. This implies 
they are not $2$-apex, see~\cite{BBFFHL,OT}. 
Since $K_7$ and its 13 descendants are also MMIK~\cite{KS},
this means that none of these 14 graphs can be a minor of a $K_{3^2,1^2}$ 
family graph.

So far, we have established that each of the size 22 graphs in the $K_{3^2,1^2}$ family is not $2$-apex. If one of these were not in $\On{2}$, then it would
have to have a $2$-apex minor of size at most 21. By \cite{BM}, this minor
would lie in the $K_7$ family. We've already argued that 14 of the 20 graphs
in that family cannot play this role. We omit the straightforward verification
that none of the remaining six graphs in the $K_7$ family is a minor of a 
member of the $K_{3^2,1^2}$ family.
\end{proof}

We verify the conjecture for $n = 3$ with the aid of a computer, as
described in Section~2 below. 
For $n = 4$, we checked that the 163 graphs in the $K_9$ family and 357 of the graphs in the $K_{3^2,1^4}$ family (all those within four $\ty$ or $\yt$ moves of
$K_{3^2,1^4}$) are in $\On{4}$. 
For $n = 5$, the computer assures us that 608 of the 1681 graphs in $K_{10}$'s family are in $\On{5}$. This gives the bounds for $\On{4}$ and $\On{5}$ mentioned earlier. Due to time constraints, we did not attempt to 
check more graphs in the $K_{10}$ family nor in the $K_{3^2,1^4}$ family of more than 13 thousand graphs.

It is straightforward to see that $K_{n+5}$ and $K_{3^2,1^n}$ themselves are 
in $\On{n}$. We also verify this for the three immediate descendants of 
each of these two graphs in the following theorem, which we prove
in Section~1. Recall that when a sequence of 
$\ty$ moves take us from $G_1$ to $G_2$, we say $G_2$ is a 
\dfn{descendant} of $G_1$.
\begin{theorem} 
\label{thm3kids}%
Let $n > 0$. The graphs $K_{n+5}$ and $K_{3^2,1^n}$ 
and the three 
immediate descendants of each are elements of $\On{n}$
\end{theorem}

It would be tempting to argue that $\ty$ and $\yt$ moves preserve the MMNnA 
property that characterizes $\On{n}$. Unfortunately, this is not true in general. For example, the disjoint unions $G_3 = \sqcup_{i=1}^{n+1} K_{3,3}$ and 
$G_5 = \sqcup_{i=1}^{n+1} K_5$ are both MMNnA. However, a $\yt$ move on 
$K_{3,3}$ results in a planar graph and, similarly, a $\yt$ move on $G_3$ gives a 
graph that is not in $\On{n}$. On the other hand, a $\ty$ move on $K_5$ results in a graph that is nonplanar, but with proper $K_{3,3}$ minor. So, the child of $G_5$ 
cannot join it in $\On{n}$.

On the other hand, it's easy to see that $\yt$ moves preserve planarity
(or, equivalently, $\ty$ preserves nonplanarity). 
If $G$ is planar, and a $\yt$ move on the degree 3 vertex $v$
results in $G'$ (We call this a \dfn{$\yt$ move at $v$}.)
then, fixing a planar embedding of $G$, we can effect the 
$\yt$ move in a tubular neighborhood of $v$ and its three edges to give
a planar embedding of $G'$. 

Using this idea we can show that a $\yt$ move at $v$ preserves the $n$-apex
property provided $v$ is not an apex. For $1$-apex graphs,
this observation is due to Barsotti in unpublished work. For $W \subset V(G)$,
let $G-W$ denote the induced subgraph on $V(G) \setminus W$.
\begin{lemma}
\label{lemnapex}%
Suppose $G$ is $n$-apex with subset $U$ of $V(G)$ such that
$|U| \leq n$ and $G-U$ is planar. If $G'$ is obtained from $G$
by a $\yt$ move at the vertex $v \not\in U$   
then $G'$ is also $n$-apex. 
\end{lemma}
\begin{proof}
The key observation is that the $\yt$ move preserves planarity 
so that $G'-U$ is again planar and 
$G'$ is $|U|$-apex. The details depend on how many of $v$'s neighbors, $a$, $b$,
and $c$, are in $U$. 

If none are, the argument is the
same as above: we can effect the $\yt$ move in a tubular neighborhood
of $v$ and its edges. If exactly one is, say $a \in U$, then the $\yt$ move
replaces the edges $vb$ and $vc$ in $G-U$ with the edge $bc$ in $G'-U$ so 
that planarity is preserved. Similarly, if two or all three neighbors are in $U$,
the $\yt$ move will leave $G'-U$ planar.
\end{proof}
Although $\yt$ and $\ty$ moves do not preserve MMNnA in general, 
our conjecture asserts that they do in the $K_{n+5}$ and $K_{3^2,1^n}$ family.

We describe results for $\On{2}$ and $\On{3}$ in Section~2.
This includes classifications of all the graphs 
in $\On{2}$ through order 9 and size 23. 
We also give a proof ``by hand" that the 32 graphs in the 
$K_8$ family are in $\On{3}$. This, with the computer verification for 
the $K_{3^2, 1^3}$ family, shows that our conjecture holds for $n = 3$.
But first we prove Theorem~\ref{thm3kids} in the next section.

\section{Proof of Theorem~\ref{thm3kids}}

In this section we prove that $K_{n+5}$, $K_{3^2,1^n}$, and their three immediate
descendants are in $\On{n}$.
In the introduction, we gave the proof for $n \leq 2$,
so we will assume $n \geq 3$. Then, the three descendants of $K_{3^2,1^n}$ are
its \dfn{children}, the graphs obtained by a single $\ty$ move. The graph
$K_{n+5}$ has a single child, $C_1$, which in turn has two children. This child and
pair of grandchildren are the three immediate descendants of $K_{n+5}$ referred to
in the statement of Theorem~\ref{thm3kids}. 

\begin{figure}[ht]
 \centering
 \includegraphics[scale = 0.5]{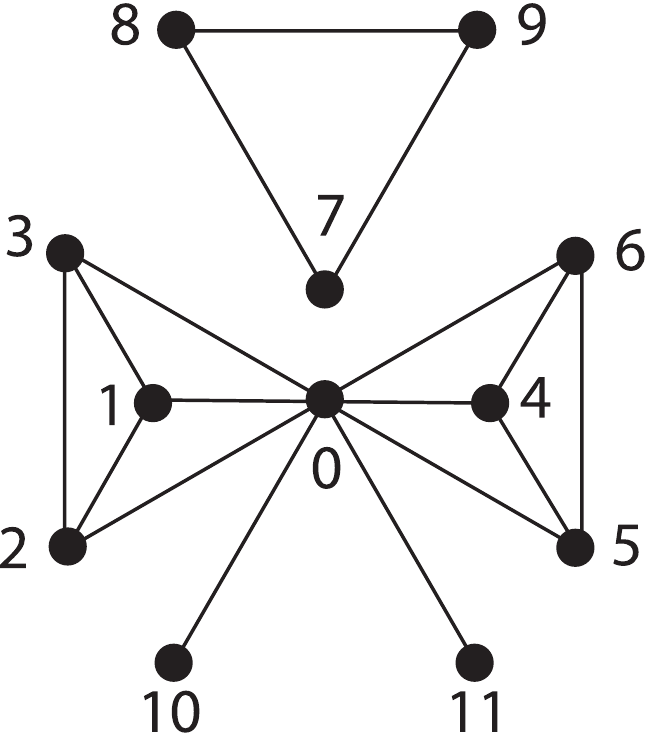}
 \caption{The complement of the graph $D''_3$.
 \label{figD3}%
 }
\end{figure}

We begin with a lemma that will be useful for one of the children of $K_{3^2,1^n}$ and which gives the flavor of our approach.
\begin{lemma}
\label{lemD3}%
Let $D''_3$ denote the complement of the $(12, 17)$ graph of Figure~\ref{figD3}
and $D'_3$ the induced subgraph on $V(D''_3) \setminus \{v_{10},v_{11}\}$. 
Then $D'_3 \in \On{3}$ and $D''_3 \in \On{5}$.
\end{lemma}
Since $D'_3$ is a child of $K_{3^2,1^3}$ and $D''_3$ is a child of $K_{3^2,1^5}$, the lemma is consistent with our conjecture.

\begin{proof}
We argue $D'_3$ is not $3$-apex by showing how deleting any
three vertices results in a nonplanar graph. 
Deleting $v_0$ yields $D'_3 - v_0 = K_{3^3}$, which is not $2$-apex. So we can assume that
$v_0$ is not among the three deleted vertices. On the other hand, 
if none of $\{v_7, v_8, v_9\}$ are deleted, then a $K_{3,3}$ subgraph 
remains after deleting three vertices from $\{v_1, v_2, v_3, v_4, v_5, v_6\}$.
Yet, if $\{v_7, v_8, v_9\}$ are the three deleted vertices, again a 
$K_{3,3}$ subgraph remains. So, we are left with the cases where exactly 
one or two vertices in $\{v_7, v_8, v_9\}$ is deleted. 

Suppose then that $v_9$ is not deleted. Since we are assuming $v_0$ is also 
not deleted, we can contract $v_0v_9$. Then, deleting $v_7$ leaves the
graph $K_{3^2,1^2}$, which is not $2$-apex. This completes the
argument that $D'_3$ is not $3$-apex.

To show that $D'_3$ is minor minimal not $3$-apex, 
we argue that for any $e \in E(D'_3)$, both deletion,
$D'_3-e$, and contraction, $D'_3/e$, of that edge gives a $3$-apex graph. Up to symmetry there are three 
types of edges: $v_0v_7$, $v_1v_4$, $v_1v_7$. In each case, we list three
vertices in an apex set.

\begin{table}[h]
   \begin{center}
      \begin{tabular}{c|c}
         Edge $e$   & Apex Sets  \\ \hline
         $v_0v_7$ & $D'_3-e$: $v_1, v_4, v_8$; $D'_3/e$:  $\overline{v_0}, v_1, v_4$ \\ \hline
         $v_1v_4$ & $D'_3-e$: $v_2, v_3, v_7$; $D'_3/e$:  $v_2, v_3, v_7$ \\ \hline
         $v_1v_7$ & $D'_3-e$: $v_4,v_5,v_8$; $D'_3/e$:  $v_2, v_3, v_8$ \\ \hline
      \end{tabular}
   \end{center}
\end{table}

For example, when $e = v_0v_7$, $D'_3-e$ becomes planar on deletion of
vertices $v_1$, $v_4$, and $v_8$. We use $\overline{v_0}$ to denote
the vertex of $D'_3/v_0v_7$ that comes from identifying $v_0$ and $v_7$.
This complete the proof that $D'_3 \in \On{3}$.

Our argument for $D''_3$ is similar and we make use of what was proved for 
the subgraph $D'_3$. We show that deleting five vertices from 
$D''_3$ leaves
a nonplanar graph. If both $v_{10}$ and $v_{11}$ are among the five, this is true 
as deleting those two gives $D'_3$, which is not $3$-apex as we just showed.
This leaves two cases: either exactly one or else neither of $v_{10}$ and $v_{11}$
is deleted.

If neither is deleted, note that those two together with one vertex each from
$\{v_1, v_2, v_3\}$, $\{v_4, v_5, v_6\}$, and $\{v_7, v_8, v_9\}$ induce 
a $K_5$ subgraph. So, we can assume that one of those triples 
is entirely deleted.
For example, say the five deleted vertices include $\{v_1, v_2, v_3\}$.
Then at most two of $\{v_7, v_8, v_9\}$ are deleted. Say $v_7$ is not. Contracting
$v_0v_7$ gives a subgraph of $K_{3^2, 1^2}$ which is not $2$-apex. 
Suppose instead that $\{v_7, v_8, v_9\}$ is included among the five deleted vertices.
If $v_0$ is also deleted, what remains is $K_{3^2,1^2}$, which not only fails to be $1$-apex 
(which is all that is required for our argument), but is in fact not $2$-apex. 
If $v_0$ is not among the deleted vertices, then, up to symmetry, there are only two ways
to choose the last two vertices to delete: either $\{v_1, v_2\}$ or $\{v_1,v_4\}$. In either case, a nonplanar graph results.

To complete the argument that $D''_3$ is not $5$-apex, suppose that exactly one
of $v_{10}$ and $v_{11}$ is deleted, say $v_{11}$. If we, in addition, delete $v_0$,
what remains is a $K_{3^3,1}$ graph. There are then three vertices left to delete, and if any pair are in the same part of $K_{3^3,1}$, after deleting that pair, we
have the not $1$-apex graph $K_{3^2,1}$ as a subgraph and a nonplanar graph
on deleting the fifth and final vertex. So, we can assume that the remaining 
three deleted vertices are taken one each from different parts of $K_{3^3,1}$.
For example, we could delete $v_1$, $v_4$, and $v_7$. This leaves a nonplanar
graph. 

This means we can assume $v_0$ is not among the five deleted vertices, still
under the assumption that $v_{11}$, but not $v_{10}$, is deleted. 
If $v_1$ and $v_2$ are deleted, the remaining graph has a 
$K_{3^2,1^2}$ subgraph and is not $2$-apex.
Using the symmetry of the graph, this means we can assume $v_7$ and $v_8$
are deleted instead. However, this again results in a graph that is not $2$-apex
by virtue of having a $K_{3^2,1^2}$ subgraph. This completes the
argument that $D''_3$ is not $5$--apex.

We next demonstrate that, for any edge $e$, both $D''_3-e$ and $D''_3/e$ are
$5$-apex. If $e$ is not incident on $v_{10}$ or $v_{11}$, we may delete those
two vertices and recognize $e$ as an edge in $D'_3$. As we have already 
found three vertex apex sets for $D'_3-e$ and $D'_3/e$, taking the union
with $\{v_{10},v_{11}\}$ gives a five vertex apex set for
$D''_3-e$ or $D''_3/e$.

So, we may assume that $e$ is incident on $v_{10}$, $v_{11}$ or both. 
Up to symmetry,  there
are three such edges, and, in each case, we provide a five vertex apex set. 

\begin{table}[h]
   \begin{center}
      \begin{tabular}{c|c}
         Edge $e$   & Apex Sets  \\ \hline
         $v_1v_{10}$ & $D''_3-e$: $v_4, v_5, v_6, v_7, v_8$; 
         $D''_3/e$:  $v_2, v_3, v_4, v_7, v_8$ \\ \hline
         $v_7v_{10}$ & $D''_3-e$: $v_1, v_4, v_8, v_9, v_{11}$; 
         $D''_3/e$:  $v_1, v_2, v_3, v_4, v_8$ \\ \hline
         $v_{10}v_{11}$ & $D''_3-e$: $v_1,v_4,v_7,v_8,v_9$; 
         $D''_3/e$:  $v_1,v_2, v_3, v_4,v_7$ \\ \hline
      \end{tabular}
   \end{center}
\end{table}

This completes the proof that $D''_3 \in \On{5}$.
\end{proof}

\begin{proof} (of Theorem~\ref{thm3kids})
We begin with $K_{n+5}$, which is not $n$-apex since deletion of any $n$ vertices gives the nonplanar graph $K_5$. On the other hand, given $e \in E(K_{n+5})$, both
deletion and contraction of $e$ result in an $n$-apex graph as can be seen by 
deleting any $n$ vertices not on the edge $e$. Therefore,
every proper minor of $K_{n+5}$ is $n$-apex and $K_{n+5} \in \On{n}$.

Let $C_1$ denote the child of $K_{n+5}$. This graph has a unique vertex, $v_0$, of degree three with neighbors $v_1$, $v_2$, and $v_3$. 
Denote the remaining vertices $v_4, v_5, v_6, \ldots, v_{n+5}$.
Generally, deleting $n$ vertices results in a graph with $K_{3,3}$ minor. The exception is when exactly
one of $v_1$, $v_2$, and $v_3$ is deleted and $v_0$ is not. In this case, the result has a $K_5$ minor. Thus, $C_1$ is not $n$-apex.

Up to symmetry, there are three types of edges in $C_1$, represented by $v_0 v_1$, $v_1 v_4$, and $v_4 v_5$.
For each choice of edge, $e$, we argue that both
deletion, $C_1 - e$, and contraction, $C_1/e$, gives an $n$-apex graph.

For $C_1/e$, we observe that, no matter which edge is contracted, by appropriate deletion of $n$ vertices, 
we can ensure that $v_0$ (or the vertex formed by identifying 
$v_0$ and $v_1$ in case $e = v_0v_1$) is an isolated vertex in the resulting
graph on $5$ vertices, which is, therefore, nonplanar.
If $e = v_0 v_1$ or $v_1v_4$, $C_1-e$ is planar after deleting all vertices but $v_0$ through $v_6$, while $C_1 - v_4v_5$ becomes planar after 
deletion of at least one of $v_1, v_2, v_3$ and none of $v_0, v_4, v_5$.

Therefore $C_1 \in \On{n}$ and we turn to its children. To form a triangle in 
$C_1$, we use at most one vertex from $v_1$, $v_2$, and $v_3$. This means
a triangle cannot include $v_0$ and, up to symmetry, there are two types
of triangles: $v_1v_4v_5$, resulting in the child $C_2$, or $v_4v_5v_6$, which
gives a child we'll call $C_3$. For these new graphs, we'll use the same vertex labels as in $C_1$ with the addition of a new degree three vertex $w_0$.

Let's see why $C_2 \in \On{n}$. First, notice that there is a graph automorphism by 
interchanging $v_0$ and $w_0$ and the pair $\{v_2,v_3\}$ with $\{v_4,v_5\}$.
Deleting $n-2$ of the $n$ vertices in $\{v_6, v_7, \ldots v_{n+5} \}$
results in a graph in the $K_7$ family, that is not $2$-apex 
by Proposition~2. The only way $C_2$ could remain planar 
after deletion of $n$ vertices is if at least three of those vertices remain.
However the remaining vertices from that set, together with 
up to one each of $\{v_2,v_3\}$ and $\{v_4, v_5\}$, form a clique. 
A clique on five or more vertices contains $K_5$ and is nonplanar.
Thus, we surely have a nonplanar subgraph by deletion of $n$ vertices 
except, possibly, in two cases (up to symmetry): 
1) $\{v_6, v_7, v_8, v_9\}$ remain after deleting the $n$ vertices, or
2) $\{v_6, v_7, v_8\}$ and vertices from at most one of $\{v_2, v_3\}$ and
$\{v_4, v_5 \}$, say the first, remain.

In the first case, we see that $\{v_2, v_3, v_4, v_5, v_{10}, v_{11}, \ldots v_{n+5}\}$ are the $n$ deleted vertices and what remains is the graph induced
on $v_0$, $v_1$, $v_6$, $v_7$, $v_8$, $v_9$, and $w_0$, which is nonplanar.

In the second case, we can assume $v_2$ remains as otherwise we must have 
deleted the $n+1$ vertices $\{v_2, v_3, v_4, v_5, v_9, v_{10}, \ldots, v_{n+5}\}$.
Since $v_2$ remains, we
see that $\{v_4, v_5, v_9, v_{10}, \ldots, v_{n+5}\}$ 
are $n-1$ of the deleted vertices. The graph $H$ that remains after deleting these 
is the induced graph on the eight vertices 
$v_0$, $v_1$, $v_2$, $v_3$, $v_6$, $v_7$, $v_8$, and $w_0$. 
It is straightforward to verify that $H$ is not $1$-apex. 
So, no matter which vertex is chosen as the $n$th for deletion, the result 
will be nonplanar in this case. This completes the argument that $C_2$ is not $n$-apex.

Next we show that every proper minor is $n$-apex by establishing this for
$C_2 - e$ and $C_2/e$ whatever edge $e$ is chosen. Up to symmetry, 
there are six choices for $e$: $v_0v_1$, $v_0v_2$, $v_1v_6$, $v_2v_4$,
$v_2v_6$,  and $v_6v_7$. Recall that deleting $n-2$ vertices from
$\{ v_6, v_7, \ldots, v_{n+5} \}$ converts $C_2$ to a graph $C'_2$ in the $K_7$ family.
Each choice of $e$ corresponds to an edge in $C'_2$ and we've already
established, in Proposition~2, that $C'_2 - e$ and $C'_2/e$ are $2$-apex. 
This means the corresponding minors of $C_2$ are $n$-apex. This completes 
the argument that $C_2 \in \On{n}$.

Let's see why $C_3$ is not $n$-apex. Deleting $n-2$ vertices in $V_{7+} = \{ v_7, v_8, \ldots, v_{n+5} \}$ results in a graph in the $K_7$ family that is not $2$-apex by Proposition~2. 
So, to achieve a planar graph after deleting $n$ vertices would require at least 
two vertices in $V_{7+}$ to remain. On the other hand the vertices in $V_{7+}$
induce a clique so, to avoid the nonplanar $K_5$, at most four can remain. 
Up to symmetry, this gives three cases:
1) $\{v_7, v_8, v_9, v_{10}\}$, 2) $\{v_7, v_8, v_9\}$, or 3) $\{v_7, v_8\}$ 
are the elements that remain after deleting $n$--vertices. To complete the 
argument that $C_3$ is not $n$-apex, we'll show that none of these lead 
to a planar graph.

Any vertex from $\{v_1,v_2,v_3\}$ or $\{v_4, v_5, v_6\}$ will also induce a clique
with vertices from $V_{7+}$. In the first case, if we are to avoid a nonplanar graph, 
we must delete
the $n+1$ vertices $\{ v_1, v_2, \ldots v_6, v_{11}, v_{12}, \ldots, v_{n+5} \}$, 
so we cannot use this case to show $C_3$ is $n$-apex.
In the second case, we could have vertices from at most one of $\{v_1, v_2, v_3 \}$ and $\{v_4, v_5, v_6\}$, say the first, in addition to $\{v_7, v_8, v_9\}$. This means 
we have $\{v_4, v_5, v_6, v_{10}, v_{11}, \ldots, v_{n+5} \}$ as $n-1$ of the deleted vertices. It's straightforward to verify that the graph induced on the other vertices, 
$v_0, v_1, v_2, v_3, v_7, v_8, v_9, w_0$ is not 1-apex. For the final case,
$\{ v_9, v_{10},  \ldots, v_{n+5} \}$ are $n-3$ of the deleted vertices. The graph 
induced on the other vertices, $v_0, v_1, \ldots v_8, w_0$ is not $3$-apex. Thus, in any case, we deduce that $C_3$ is not $n$-apex.

Up to symmetry, there are four types of edges in $C_3$: 
$v_0v_1$, $v_1v_4$, $v_1v_7$, and $v_7v_8$. By deleting $n-2$ vertices 
in $\{v_7, v_8, \ldots, v_{n+5} \}$ we arrive at the graph $C'_3$ in the $K_7$ family.
Aside from $v_7v_8$, each choice of edge $e$ is represented in $C'_3$. In Proposition~2 we argued that $C'_3-e$ and $C'_3/e$ are $2$-apex, which means 
$C_3-e$ and $C_3/e$ are $n$-apex. When $e=v_7v_8$, we can think of deleting 
the $n-3$ vertices $\{ v_9, v_{10}, \ldots, v_{n+5} \}$ 
to give the graph $H$ induced on the vertices $v_0, v_1, \ldots v_8, w_0$, a graph
in the $K_8$ family. Now, deleting vertices $v_1, v_2, v_4$ 
shows that $H-v_7v_8$ is $3$-apex while deleting vertices $v_1, v_2, v_3$ demonstrates that $H/v_7v_8$ is also $3$-apex. It follows that deleting or contracting $v_7v_8$ in $C_3$ results in an $n$-apex graph. Since any 
$C_3-e$ or $C_3/e$ is $n$-apex, it follows that any proper minor is $n$-apex. 
This completes the argument that $C_3 \in \On{n}$.

The graph $K_{3^2,1^n}$ is not $n$-apex as deleting any $n$ vertices 
leaves a graph with $K_{3,3}$ minor.
Let $v_1, v_2, v_3$ and $v_4, v_5, v_6$ denote the vertices
in the two parts with three vertices and $v_7, v_8, \ldots v_{n+6}$ the vertices
in the remaining $n$ parts. Up to symmetry, there are three types of 
edges $e$: $v_1v_4$, $v_1v_7$, and $v_7v_8$.  Removing $n$ vertices from
$K_{3^2,1^n}/e$ gives a graph on five vertices which is planar as long as it is 
not $K_5$. For example, ensuring that $v_5$ and $v_6$ are among the five
vertices that remain shows that $K_{3^2,1^n}/e$ is $n$-apex. On the other hand, 
deleting $n$ vertices from $K_{3^2, 1^n} - e$ gives a graph on six vertices. 
For each choice of $e$, we list six vertices that induce a planar subgraph:
for $e = v_1v_4$, use $v_1, v_2, \ldots, v_6$; 
if $e = v_1v_7$, take $v_1, v_2, v_3, v_4, v_5, v_7$; and when 
$e = v_7v_8$, then $v_1, v_2, v_4, v_5, v_7, v_8$ will work. This shows that any
proper minor of $K_{3^2,1^n}$ is $n$-apex and completes the proof that it is in $\On{n}$.

Up to symmetry, there are three types of triangles in $K_{3^2, 1^n}$. We'll 
call the three children $D_1$, $D_2$, and $D_3$, corresponding to $\ty$ moves on
on $v_1v_4v_7$, $v_1v_7v_8$, and 
$v_7v_8v_9$ respectively. We'll argue each child is in $\On{n}$.
In $D_i$, we preserve the vertex labels $v_1, v_2, \ldots, v_{n+6}$ 
from $K_{3^2,1^n}$ and call the new vertex of degree three $v_0$.

In $D_1$, deleting $n-2$ vertices from 
$V_{8+} = \{v_8, v_9, \ldots, v_{n+6}\}$ gives a graph, $D'_1$,
in the $K_{3^1, 1^2}$ family that is not $2$-apex by Proposition~2. 
So, at least two vertices from $V_{8+}$ must remain. On the other hand,
vertices in $V_{8+}$ induce a clique so at most four can remain and we again 
have three cases: 1) $\{ v_8, v_9, v_{10}, v_{11} \}$, 2) $\{v_8, v_9, v_{10}\}$, 
and 3) $\{v_8, v_9\}$ are the vertices of $V_{8+}$ that remain.

In addition, up to one vertex each from $\{v_2, v_3\}$ and $\{v_5, v_6\}$
can be added to those remaining from $V_{8+}$ to make a clique. So, in the first
case, we can assume that $n-1$ vertices were deleted leaving the graph 
induced on $v_0, v_1, v_4, v_7,  v_8, v_9, v_{10}, v_{11}$, which is not 
$1$-apex. In the second case, let's say that vertices from $\{v_2, v_3\}$ may remain.
Then, after deleting $n-2$ vertices, we have the graph induced on
$v_0, v_1, v_2, v_3, v_4, v_7, v_8, v_9, v_{10}$, which is not $2$-apex.
Finally, in the third case, after deleting $n-3$ vertices, 
we have the graph induced on $v_0, v_1, \ldots , v_9$, which is not $3$-apex. 
This shows that $D_1$ is not $n$-apex.

Up to symmetry there are nine types of edges in $D_1$:
$v_0v_1$, $v_0v_7$, $v_1v_5$, $v_1v_8$, $v_2v_5$, $v_2v_7$, 
$v_2v_8, v_7v_8$, and $v_8v_9$.
Except for the last one, $v_8v_9$, each edge $e$ is represented in 
$D'_1$, a graph in the $K_{3^2, 1^2}$ family
obtained by deleting $n-2$ vertices from $V_{8+}$. In Proposition~2, we argued
that $D'_1-e$ and $D'_1/e$ are $2$-apex and it follows that $D_1-e$ and $D_1/e$
are $n$-apex when $e \neq v_8v_9$. If we delete $n-3$ vertices in $V_{8+}$, 
leaving $\{v_8,v_9\}$, we obtain the graph $H$ induced on  $v_0, v_1, \ldots , v_9$, which includes the edge $v_7v_8$. Deleting $v_1, v_4, v_7$ shows that $H - v_8v_9$ is $3$-apex and $D_1 - v_8v_9$ is $n$-apex. For $H/v_8v_9$, deletion of $v_1, v_2,
v_3$ shows that it it $3$-apex and $D_1/v_8v_9$ is $n$-apex. We've shown that for any
edge $e$, $D_1-e$ and $D_1/e$ are both $n$-apex. It follows that any proper
minor is $n$-apex, which completes the argument that $D_1 \in \On{n}$.

The verification for $D_2$ is similar 
and we'll omit some details. Deleting the $n-2$ vertices in $V_{9+}$ results
in the graph $D'_2$ in the $K_{3^2,1^2}$ family, which is not $2$-apex. 
This leaves the cases where between four and one vertex of $V_{9+}$ remain. 
Since one vertex each from $\{v_2, v_3\}$, $\{v_4, v_5, v_6\}$, and $\{v_7,v_8\}$ can
join with vertices of $V_{9+}$ to form a clique, some or all of these must be deleted
to avoid $K_5$, a clique on five vertices, surviving. For example, in the case that
four vertices of $V_{9+}$ remain, then all seven vertices $v_2, v_3 \ldots v_8$ 
must be deleted in addition to the $n-6$ taken from $V_{9+}$, which is a total of
$n+1$ vertices. In the case that only one vertex of $V_{9+}$ remains, we have removed $n-3$ vertices from $V_{9+}$ and argue that the graph that remains, 
on ten vertices, is not $3$-apex. The remaining cases can be approached in the 
same way; remove either $n-4$ or $n-5$ vertices from $V_{9+}$ and argue that 
the remaining graphs are not $4$-apex nor $5$-apex. This shows $D_2$ is not
$n$-apex.

There are several edge types in $D_2$, but with the exception of those
incident on vertices in $V_{9+}$, all are represented in $D'_2$. For edges 
$e$ incident on one or two vertices of $V_{9+}$, we argue directly that the 
graphs that remain after deleting $n-3$ or $n-4$ vertices
of $V_{9+}$ and either contracting or deleting $e$ are 
$3$-apex or $4$-apex, respectively. This implies
that the corresponding graph $D_2 -e$ or $D_2/e$ is $n$-apex, as required.

Finally, we turn to $D_3$. As with earlier cases, we will leverage a similar graph $D'_3$, this time in the $K_{3^2,1^3}$ family,
In Lemma~\ref{lemD3} we proved $D'_3 \in \On{3}$. Deleting the $n-3$ vertices in 
$V_{10+}$ gives $D'_3$, so to prove $D_3$ is not $n$-apex it remains to take the cases
where between one and four vertices of $V_{10+}$ remain.

Taking a vertex each from $\{v_1, v_2, v_3\}$, $\{v_4, v_5, v_6\}$,
and $\{v_7, v_8, v_9\}$ gives a clique with any vertices that remain from $V_{10+}$.
In case four vertices from $V_{10+}$ remain, this means deletion of a further nine vertices
in addition to the $n-7$ from $V_{10+}$ to avoid a $K_5$ minor and so cannot be used to 
show $D_3$ is $n$-apex. Suppose the three vertices $\{v_{10}, v_{11}, v_{12}\}$ remain
from $V_{10+}$. To avoid a $K_5$ minor, we must delete two of the triples and there are a
couple of ways to do this. If $\{v_1, v_2, \ldots, v_6\}$ are deleted, then $n$ vertices are
gone and we observe that the resulting graph induced on $v_0, v_7, v_8, \ldots, v_{12}$ is nonplanar.
On the other hand, if $\{v_4, v_5, \ldots, v_9\}$ are deleted, the result, 
induced on $v_0, v_1, v_2, v_3, v_{10}, v_{11}, v_{12}$ is again nonplanar.
So, we can assume at most two vertices from $V_{10+}$ remain.
If two, we've deleted $n-5$ vertices to leave the graph $D''_3 \in \On{5}$ (see Lemma~\ref{lemD3})
and if only one, what's left is a graph in the $K_{3^2,1^6}$ family that can be shown to 
be not $6$-apex. This completes the argument that $D_3$ is not $n$-apex.

Finally, any $e \in E(D_3)$ is realized as an edge of $D''_3$. In Lemma~\ref{lemD3}, we prove that 
$D''_3-e$ and $D''_3/e$ are $5$-apex and it follows that $D_3$ and $D''_3$ are $n$-apex,
completing the proof that $D_3 \in \On{n}$.
\end{proof}

\section{Results for $\On{2}$ and $\On{3}$}

In this section we prove Conjecture~1 when $n=3$. However, we begin
with classifications of the graphs in $\On{2}$
through order nine and size 23.
Mattman~\cite{M} established that the elements of $\On{2}$ have size
21 at least and Barsotti and Mattman~\cite{BM} showed that the $K_7$ family is
precisely the set of graphs of size 21. The 58 $K_{3^2,1^2}$ graphs are
of size 22 in $\On{2}$ and, based on a computer search, we have found
exactly two other graphs with 22 edges. Both are $4$--regular on 11 vertices
and we list their edges below.

$$
\{\{1,6\},\{1,9\},\{1,10\},\{1,11\},\{2,6\},\{2,9\},\{2,10\},\{2,11\},
$$ 
$$
\{3,7\},\{3,8\},\{3,9\},\{3,10\},\{4,7\},\{4,8\},\{4,9\},
$$
$$
\{4,11\},\{5,7\},\{5,8\},\{5,10\},\{5,11\},\{6,7\},\{6,8\}\}
$$

$$
\{\{1,5\},\{1,6\},\{1,7\},\{1,8\},\{2,6\},\{2,7\},\{2,8\},\{2,9\},$$
$$
\{3,7\},\{3,8\},\{3,9\},\{3,10\},\{4,8\},\{4,9\},\{4,10\},
$$
$$
\{4,11\},\{5,9\},\{5,10\},\{5,11\},\{6,10\},\{6,11\},\{7,11\}\}
$$

We were surprised to discover that there are {\bf no} graphs in $\On{2}$ of 
size 23. This is similar to $\On{1}$ where the Petersen family 
graphs of size 15 are the only obstructions with fewer than 18 edges (see
\cite{BM}). On the other hand, we do have examples of graphs in $\On{2}$ for
each size between 24 and 30, inclusive. 

In terms of order, we found a total of 12 graphs in $\On{2}$ with nine or fewer
vertices. There are five each in the $K_{7}$ and $K_{3^2,1^2}$ families and 
two of size 26 and 27 whose edge lists we give below.
$$\{\{1, 4\}, \{1, 5\}, \{1, 7\}, \{1, 8\}, \{1, 9\}, \{2, 5\}, \{2, 6\}, \{2, 7\}, \{2, 
  8\},$$
 $$   \{2, 9\}, \{3, 5\}, \{3, 6\},
  \{3, 7\}, \{3, 8\}, \{3, 9\}, \{4, 6\}, \{4, 
  7\}, \{4, 8\}, $$
$$  \{4, 9\}, \{5, 6\}, \{5, 8\}, \{5, 9\}, \{6, 8\}, \{6, 9\}, \{7, 
  8\}, \{7, 9\}\}$$
  
$$ \{\{1, 4\}, \{1, 5\}, \{1, 6\}, \{1, 7\}, \{1, 8\}, \{1, 9\}, \{2, 4\}, \{2, 5\}, \{2, 
  6\}, $$
$$  \{2, 7\}, \{2, 8\}, \{2, 9\}, \{3, 4\}, \{3, 5\}, \{3, 6\}, \{3, 7\}, \{3, 
  8\}, \{3, 9\},$$
$$ \{4, 6\}, \{4, 7\}, \{4, 8\}, \{5, 7\}, \{5, 8\}, \{5, 9\}, \{6, 
  8\}, \{6, 9\}, \{7, 9\}\}$$
In total, this gives 82 graphs in $\On{2}$: 20 in the $K_7$ family,
58 in $K_{3^2,1^2}$'s, and four more with edges listed above.

\begin{figure}[ht]
 \centering
\includegraphics[scale=0.5]{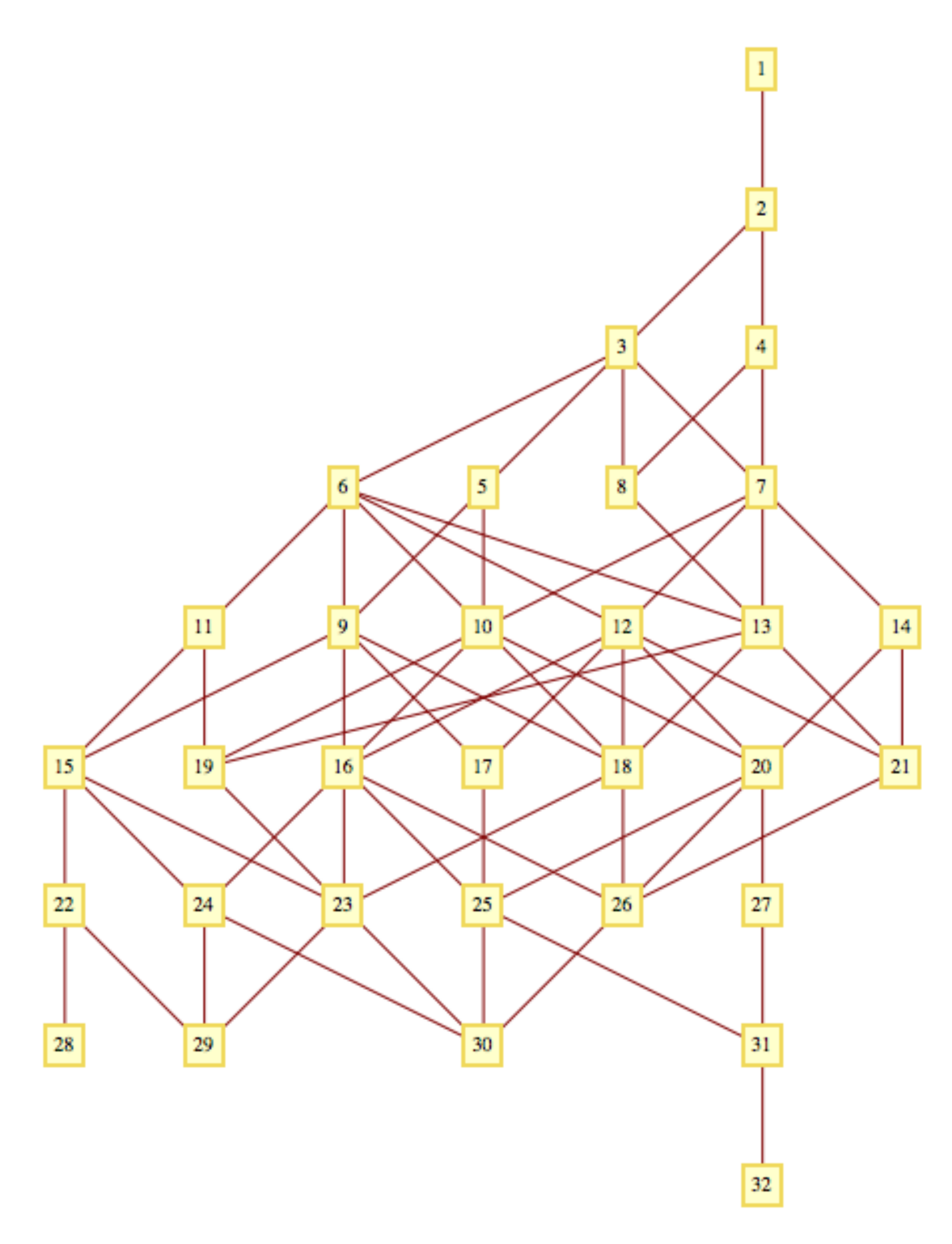}
 \caption{The $K_8$ family.
 \label{figK8fam}%
 }
\end{figure}

Next, we show that Conjecture~1 holds for $n = 3$. For this we rely on
a computer verification that the 569 graphs in the $K_{3^2, 1^3}$ family are
MMN3A. Then the following theorem completes the proof. Together with
the $K_{3^2,1^3}$ family, the 32 graphs in $K_8$'s give the bound 
$\On{3} \geq 601$ mentioned in the introduction.
\begin{theorem} The $K_8$ family is in $\On{3}$.
\end{theorem}
\begin{proof}
The $K_8$ family consists of 32 graphs, illustrated schematically in 
Figure~\ref{figK8fam}. In the appendix we give edge lists for these graphs.
Graphs at the same height are of the same order and edges join graphs that are 
related by a $\ty$ move. All graphs in this family have 28 edges.
The numbering simply reflects the order in which these graphs were 
generated by our computer program and we refer to the graphs in the family as
Cousin 1, $K_8$, through Cousin 32, a graph of order 16.

Theorem~\ref{thm3kids} shows that Cousin 1, $K_8$, and its child, 
Cousin 2, and grandchildren, Cousins 3 and 4, are all MMN3A.
We apply Lemma~4 to see that the remaining 
descendants of $K_8$ are not $3$-apex. 

For example, let $C_5$ denote Cousin 5, which is a child of Cousin 3, $C_3$.
In other words a $\yt$ move at some vertex $v$ in $C_5$ gives $C_3$. 
If $C_5$ were $3$-apex and $v$ not an apex, then Lemma~4 would
imply $C_3$ is $3$-apex, a contradiction. So, to show that $C_5$ is
not $3$-apex, it is enough to argue that $v$ cannot be an apex. For this, 
it suffices to show that $C_5 - v$ is not $2$-apex. Indeed, $C_3$ can be
obtained by a $\yt$ move at vertex $v = 9$. Deleting edges $\{4,6\}$, $\{4,8\}$, and
$\{6,8\}$ from $C_5-9$ results in a graph 
(Cousin 4 by the numbering in~\cite{GMN}) in the $K_{3^2,1^2}$ family.
This shows that $C_5-9$ is not $2$-apex, as required.

In the appendix, we list, for each cousin $C_n$ ($5 \leq n \leq 32$)
in the $K_8$ family, 
a vertex $v$ such that a $\yt$ move at $v$ gives another graph in the family, and 
edges of $C_n$ such that deleting or contracting those edges from 
$C_n-v$ gives a graph in the $K_{3^2,1^2}$ family. 
In fact, we can take $v=9$ in all cases.
This shows no graphs in the family are $3$-apex.

Further, for $5 \leq n \leq 32$, we show that each proper minor of $C_n$ is $3$-apex. For each edge $e$ (up to symmetry), we show that both $C_n-e$ and 
$C_n/e$ (edge deletion and contraction, respectively) are $3$-apex by listing
three vertices whose deletion gives a planar subgraph. This completes the proof.
\end{proof}

\appendix

\section{Edge lists of $K_8$ family graphs} 
In Tables~\ref{tblK81} and \ref{tblK82}
we list the edges of the 32 graphs in the $K_8$ family. We number the graphs as Cousin 1 ($K_8$) through
Cousin 32 (a graph of order 16).

\begin{table}[hp]
   \begin{center}
      \begin{tabular}{c|c}
\label{tblK81}%

      $n$ & Edges of $C_n$ \\ \hline
1 &       
$$\{\{1, 2\}, \{1, 3\}, \{1, 4\}, \{1, 5\}, \{1, 6\}, \{1, 7\}, \{1, 8\},  \{2, 3\}, \{2, 4\}, \{2, 5\},$$ \\ 
& $$\{2, 6\}, \{2, 7\}, \{2, 8\}, \{3, 4\}, \{3, 5\}, \{3, 6\},    \{3, 7\}, \{3, 8\}, \{4, 5\},$$ \\
& $$\{4, 6\}, \{4, 7\}, \{4, 8\}, \{5, 6\}, \{5, 7\}, \{5, 8\}, \{6, 7\}, \{6, 8\}, \{7, 8\}\}$$ \\ \hline
2 &
$$\{\{1, 4\}, \{1, 5\}, \{1, 6\}, \{1, 7\}, \{1, 8\}, \{1, 9\}, \{2, 4\}, \{2, 5\}, \{2, 6\}, \{2, 7\},$$ \\
&$$\{2, 8\}, \{2, 9\}, \{3, 4\}, \{3, 5\}, \{3, 6\},    \{3, 7\}, \{3, 8\}, \{3, 9\}, \{4, 5\}, $$ \\
&$$\{4, 6\}, \{4, 7\}, \{4, 8\}, \{5, 6\}, \{5, 7\},   \{5, 8\}, \{6, 7\}, \{6, 8\}, \{7, 8\}\}$$ \\ \hline
3 &
$$\{\{1, 6\}, \{1, 7\}, \{1, 8\}, \{1, 9\}, \{1, 10\}, \{2, 4\}, \{2, 5\}, \{2, 6\}, \{2, 7\}, \{2, 8\},$$ \\
&$$ \{2, 9\}, \{3, 4\}, \{3, 5\}, \{3, 6\},    \{3, 7\}, \{3, 8\}, \{3, 9\}, \{4, 6\}, \{4, 7\}, $$ \\
&$$\{4, 8\}, \{4, 10\}, \{5, 6\}, \{5, 7\},    \{5, 8\}, \{5, 10\}, \{6, 7\}, \{6, 8\}, \{7, 8\}\}$$ \\ \hline
4 &
$$\{\{1, 4\}, \{1, 5\}, \{1, 6\}, \{1, 7\}, \{1, 8\}, \{1, 9\}, \{2, 4\}, \{2, 5\}, \{2, 6\}, \{2, 7\},$$ \\
&$$ \{2, 8\}, \{2, 9\}, \{3, 4\},    \{3, 5\}, \{3, 6\}, \{3, 7\}, \{3, 8\}, \{3, 9\}, \{4, 7\},$$ \\ 
&$$ \{4, 8\}, \{4, 10\}, \{5, 7\},    \{5, 8\}, \{5, 10\}, \{6, 7\}, \{6, 8\}, \{6, 10\}, \{7, 8\}\}$$ \\ \hline
5 &
$$\{\{1, 8\}, \{1, 9\}, \{1, 10\}, \{1, 11\}, \{2, 4\}, \{2, 5\}, \{2, 6\}, \{2, 7\}, \{2, 8\},    \{2, 9\},$$ \\
& $$ \{3, 4\}, \{3, 5\}, \{3, 6\}, \{3, 7\}, \{3, 8\}, \{3, 9\}, \{4, 6\}, \{4, 7\},    \{4, 8\}, $$ \\
& $$\{4, 10\}, \{5, 6\}, \{5, 7\}, \{5, 8\}, \{5, 10\}, \{6, 8\}, \{6, 11\}, \{7, 8\}, \{7, 11\}\}$$ \\ \hline
6 & 
$$\{\{1, 6\}, \{1, 7\}, \{1, 8\}, \{1, 9\}, \{1, 10\}, \{2, 5\}, \{2, 7\},   \{2, 8\}, \{2, 9\}, \{2, 11\},$$ \\
& $$ \{3, 4\}, \{3, 5\}, \{3, 6\}, \{3, 7\}, \{3, 8\}, \{3, 9\},    \{4, 7\}, \{4, 8\}, \{4, 10\},$$ \\
& $$ \{4, 11\}, \{5, 6\}, \{5, 7\}, \{5, 8\}, \{5, 10\}, \{6, 7\},    \{6, 8\}, \{6, 11\}, \{7, 8\}\}$$ \\ \hline
7 &
$$\{\{1, 6\}, \{1, 7\}, \{1, 8\}, \{1, 9\}, \{1, 10\}, \{2, 4\}, \{2, 5\}, \{2, 8\}, \{2, 9\}, \{2, 11\},$$ \\
& $$  \{3, 4\}, \{3, 5\}, \{3, 6\}, \{3, 7\},    \{3, 8\}, \{3, 9\}, \{4, 6\}, \{4, 7\}, \{4, 8\},$$ \\
& $$ \{4, 10\}, \{5, 6\}, \{5, 7\}, \{5, 8\},   \{5, 10\}, \{6, 8\}, \{6, 11\}, \{7, 8\}, \{7, 11\}\}$$ \\ \hline
8 &
$$ \{\{1, 6\}, \{1, 7\}, \{1, 8\},  \{1, 9\}, \{1, 10\}, \{2, 4\}, \{2, 5\}, \{2, 6\}, \{2, 7\}, \{2, 8\},$$ \\
& $$ \{2, 9\}, \{3, 4\},    \{3, 5\}, \{3, 6\}, \{3, 7\}, \{3, 8\}, \{3, 9\}, \{4, 6\}, \{4, 7\},$$ \\
& $$ \{4, 8\}, \{4, 10\},    \{5, 6\}, \{5, 7\}, \{5, 8\}, \{5, 10\}, \{6, 11\}, \{7, 11\}, \{8, 11\}\}$$ \\ \hline
9 & 
$$\{\{1, 8\}, \{1, 9\}, \{1, 10\}, \{1, 11\}, \{2, 5\}, \{2, 7\}, \{2, 8\}, \{2, 9\}, \{2, 12\},    \{3, 4\}, $$ \\
& $$\{3, 5\}, \{3, 6\}, \{3, 7\}, \{3, 8\}, \{3, 9\}, \{4, 7\}, \{4, 8\}, \{4, 10\},    \{4, 12\},$$ \\
& $$ \{5, 6\}, \{5, 7\}, \{5, 8\}, \{5, 10\}, \{6, 8\}, \{6, 11\}, \{6, 12\}, \{7, 8\}, \{7, 11\}\}$$ \\ \hline
10 &
$$\{\{1, 8\}, \{1, 9\}, \{1, 10\}, \{1, 11\}, \{2, 5\}, \{2, 6\},    \{2, 7\}, \{2, 9\}, \{2, 12\}, \{3, 4\},$$ \\
& $$ \{3, 5\}, \{3, 6\}, \{3, 7\}, \{3, 8\}, \{3, 9\},    \{4, 6\}, \{4, 7\}, \{4, 10\}, \{4, 12\}, $$ \\
& $$\{5, 6\}, \{5, 7\}, \{5, 8\}, \{5, 10\}, \{6, 8\},    \{6, 11\}, \{7, 8\}, \{7, 11\}, \{8, 12\}\}$$ \\
\hline
11 &
$$\{\{1, 6\}, \{1, 7\}, \{1, 8\}, \{1, 9\}, \{1, 10\}, \{2, 5\}, \{2, 7\}, \{2, 8\}, \{2, 9\}, \{2, 11\}, $$ \\
& $$  \{3, 4\}, \{3, 7\}, \{3, 8\},    \{3, 9\}, \{3, 12\}, \{4, 7\}, \{4, 8\}, \{4, 10\}, \{4, 11\}, $$ \\
& $$ \{5, 7\}, \{5, 8\},    \{5, 10\}, \{5, 12\}, \{6, 7\}, \{6, 8\}, \{6, 11\}, \{6, 12\}, \{7, 8\}\}$$ \\ \hline
12 &
$$ \{\{1, 6\}, \{1, 7\}, \{1, 8\}, \{1, 9\}, \{1, 10\}, \{2, 5\}, \{2, 7\}, \{2, 8\}, \{2, 9\},    \{2, 11\}, $$ \\
& $$  \{3, 4\}, \{3, 6\}, \{3, 8\}, \{3, 9\}, \{3, 12\}, \{4, 7\}, \{4, 8\}, \{4, 10\},    \{4, 11\}, $$ \\
& $$  \{5, 6\}, \{5, 8\}, \{5, 10\}, \{5, 12\}, \{6, 7\}, \{6, 8\}, \{6, 11\}, \{7, 8\}, \{7, 12\}\}$$ \\ \hline
13 &
$$\{\{1, 6\}, \{1, 7\}, \{1, 8\}, \{1, 9\}, \{1, 10\}, \{2, 5\},   \{2, 7\}, \{2, 8\}, \{2, 9\}, \{2, 11\}, $$ \\
& $$  \{3, 4\}, \{3, 5\}, \{3, 6\}, \{3, 9\}, \{3, 12\},    \{4, 7\}, \{4, 8\}, \{4, 10\}, \{4, 11\},  $$ \\
& $$ \{5, 6\}, \{5, 7\}, \{5, 8\}, \{5, 10\}, \{6, 7\},    \{6, 8\}, \{6, 11\}, \{7, 12\}, \{8, 12\}\}$$ \\ \hline
14 &
$$ \{\{1, 6\}, \{1, 7\}, \{1, 8\}, \{1, 9\}, \{1, 10\}, \{2, 4\}, \{2, 5\}, \{2, 8\}, \{2, 9\}, \{2, 11\},$$ \\
& $$ \{3, 4\}, \{3, 5\}, \{3, 6\},    \{3, 7\}, \{3, 8\}, \{3, 9\}, \{4, 7\}, \{4, 10\}, \{4, 12\}, $$ \\
& $$\{5, 6\}, \{5, 7\}, \{5, 8\},    \{5, 10\}, \{6, 11\}, \{6, 12\}, \{7, 8\}, \{7, 11\}, \{8, 12\}\}$$ \\ \hline
15 &
$$\{\{1, 8\}, \{1, 9\}, \{1, 10\}, \{1, 11\}, \{2, 8\}, \{2, 9\}, \{2, 12\}, \{2, 13\}, \{3, 4\}, \{3, 5\},$$ \\
& $$ \{3, 6\}, \{3, 7\}, \{3, 8\}, \{3, 9\}, \{4, 7\}, \{4, 8\}, \{4, 10\},    \{4, 12\}, \{5, 6\},$$ \\
& $$ \{5, 8\}, \{5, 10\}, \{5, 13\}, \{6, 8\}, \{6, 11\}, \{6, 12\}, \{7, 8\}, \{7, 11\}, \{7, 13\}\}$$ \\ \hline
16 &
$$\{\{1, 8\}, \{1, 9\}, \{1, 10\}, \{1, 11\}, \{2, 7\},    \{2, 9\}, \{2, 12\}, \{2, 13\}, \{3, 4\}, \{3, 5\},$$ \\
& $$ \{3, 6\}, \{3, 7\}, \{3, 8\}, \{3, 9\},    \{4, 7\}, \{4, 8\}, \{4, 10\}, \{4, 12\}, \{5, 6\},$$ \\
& $$ \{5, 7\}, \{5, 10\}, \{5, 13\},    \{6, 8\}, \{6, 11\}, \{6, 12\}, \{7, 8\}, \{7, 11\}, \{8, 13\}\}$$ 
     \end{tabular}
\caption{Edges Lists for graphs in $K_8$ family}      
\label{tblK81}%
   \end{center}   
\end{table}

\begin{table}[hp]
   \begin{center}
      \begin{tabular}{c|c}
            
      $n$ & Edges of $C_n$ \\ \hline
17 &
$$\{\{1, 8\}, \{1, 9\}, \{1, 10\}, \{1, 11\}, \{2, 5\}, \{2, 7\}, \{2, 8\}, \{2, 9\}, \{2, 12\},    \{3, 4\},$$ \\
& $$ \{3, 6\}, \{3, 8\}, \{3, 9\}, \{3, 13\}, \{4, 7\}, \{4, 8\}, \{4, 10\}, \{4, 12\},    \{5, 6\}, $$ \\
& $$\{5, 8\}, \{5, 10\}, \{5, 13\}, \{6, 8\}, \{6, 11\}, \{6, 12\}, \{7, 8\},  \{7, 11\}, \{7, 13\}\}$$ \\ \hline
18 & 
$$ \{\{1, 8\}, \{1, 9\}, \{1, 10\}, \{1, 11\}, \{2, 5\}, \{2, 7\},    \{2, 8\}, \{2, 9\}, \{2, 12\}, \{3, 4\}, $$ \\
& $$ \{3, 6\}, \{3, 7\}, \{3, 9\}, \{3, 13\}, \{4, 7\},    \{4, 8\}, \{4, 10\}, \{4, 12\}, \{5, 6\},$$ \\
& $$ \{5, 7\}, \{5, 10\}, \{5, 13\}, \{6, 8\},   \{6, 11\}, \{6, 12\}, \{7, 8\}, \{7, 11\}, \{8, 13\}\}$$ \\ \hline
19 & 
$$ \{\{1, 8\}, \{1, 9\}, \{1, 10\}, \{1, 11\}, \{2, 5\}, \{2, 6\}, \{2, 7\}, \{2, 9\}, \{2, 12\},    \{3, 4\},$$ \\
& $$ \{3, 6\}, \{3, 7\}, \{3, 9\}, \{3, 13\}, \{4, 6\}, \{4, 7\}, \{4, 10\}, \{4, 12\},  \{5, 6\},$$ \\
& $$ \{5, 7\}, \{5, 10\}, \{5, 13\}, \{6, 8\}, \{6, 11\}, \{7, 8\}, \{7, 11\},   \{8, 12\}, \{8, 13\}\}$$ \\ \hline
20 & 
$$ \{\{1, 8\}, \{1, 9\}, \{1, 10\}, \{1, 11\}, \{2, 5\}, \{2, 6\},    \{2, 7\}, \{2, 9\}, \{2, 12\}, \{3, 4\}, $$ \\
& $$\{3, 5\}, \{3, 7\}, \{3, 9\}, \{3, 13\}, \{4, 6\},  \{4, 7\}, \{4, 10\}, \{4, 12\}, \{5, 6\}, $$ \\
& $$\{5, 7\}, \{5, 8\}, \{5, 10\}, \{6, 11\},  \{6, 13\}, \{7, 8\}, \{7, 11\}, \{8, 12\}, \{8, 13\}\}$$ \\ \hline
21 &
$$\{\{1, 6\}, \{1, 7\}, \{1, 8\}, \{1, 9\}, \{1, 10\}, \{2, 5\}, \{2, 7\}, \{2, 8\}, \{2, 9\},  \{2, 11\}, $$ \\
& $$\{3, 4\}, \{3, 6\}, \{3, 8\}, \{3, 9\}, \{3, 12\}, \{4, 7\}, \{4, 8\}, \{4, 10\},  \{4, 11\}, $$ \\
& $$\{5, 6\}, \{5, 8\}, \{5, 10\}, \{5, 12\}, \{6, 11\}, \{6, 13\}, \{7, 12\},  \{7, 13\}, \{8, 13\}\}$$ \\ \hline
22 &
$$ \{\{1, 8\}, \{1, 9\}, \{1, 10\}, \{1, 11\}, \{2, 8\}, \{2, 9\},  \{2, 12\}, \{2, 13\}, \{3, 5\}, \{3, 6\},$$ \\
& $$ \{3, 8\}, \{3, 9\}, \{3, 14\}, \{4, 8\},  \{4, 10\}, \{4, 12\}, \{4, 14\}, \{5, 6\}, \{5, 8\}, $$ \\
& $$\{5, 10\}, \{5, 13\}, \{6, 8\},  \{6, 11\}, \{6, 12\}, \{7, 8\}, \{7, 11\}, \{7, 13\}, \{7, 14\}\}$$ \\ \hline
23 &
$$  \{\{1, 8\}, \{1, 9\}, \{1, 10\}, \{1, 11\}, \{2, 8\}, \{2, 9\}, \{2, 12\}, \{2, 13\},  \{3, 5\}, \{3, 6\}, $$ \\
& $$\{3, 7\}, \{3, 9\}, \{3, 14\}, \{4, 7\}, \{4, 10\}, \{4, 12\},  \{4, 14\}, \{5, 6\}, \{5, 8\}, $$ \\
& $$\{5, 10\}, \{5, 13\}, \{6, 8\}, \{6, 11\}, \{6, 12\},  \{7, 8\}, \{7, 11\}, \{7, 13\}, \{8, 14\}\}$$ \\ \hline
24 &
$$ \{\{1, 8\}, \{1, 9\}, \{1, 10\}, \{1, 11\},  \{2, 8\}, \{2, 9\}, \{2, 12\}, \{2, 13\}, \{3, 4\}, \{3, 5\},$$ \\
& $$ \{3, 6\}, \{3, 7\}, \{3, 8\},  \{3, 9\}, \{4, 10\}, \{4, 12\}, \{4, 14\}, \{5, 6\}, \{5, 8\}, $$ \\
& $$\{5, 10\}, \{5, 13\},  \{6, 8\}, \{6, 11\}, \{6, 12\}, \{7, 11\}, \{7, 13\}, \{7, 14\}, \{8, 14\}\}$$ \\ \hline
25 & 
$$ \{\{1, 8\}, \{1, 9\}, \{1, 10\}, \{1, 11\}, \{2, 7\}, \{2, 9\}, \{2, 12\}, \{2, 13\},  \{3, 4\}, \{3, 6\},$$ \\
& $$ \{3, 8\}, \{3, 9\}, \{3, 14\}, \{4, 7\}, \{4, 8\}, \{4, 10\}, \{4, 12\},  \{5, 6\}, \{5, 10\},$$ \\
& $$ \{5, 13\}, \{5, 14\}, \{6, 8\}, \{6, 11\}, \{6, 12\}, \{7, 8\},  \{7, 11\}, \{7, 14\}, \{8, 13\}\}$$ \\ \hline
26 & 
$$ \{\{1, 8\}, \{1, 9\}, \{1, 10\}, \{1, 11\}, \{2, 7\},  \{2, 9\}, \{2, 12\}, \{2, 13\}, \{3, 4\}, \{3, 5\},$$ \\
& $$ \{3, 6\}, \{3, 9\}, \{3, 14\}, \{4, 7\},  \{4, 8\}, \{4, 10\}, \{4, 12\}, \{5, 6\}, \{5, 7\}, $$ \\
& $$\{5, 10\}, \{5, 13\}, \{6, 8\},  \{6, 11\}, \{6, 12\}, \{7, 11\}, \{7, 14\}, \{8, 13\}, \{8, 14\}\}$$ \\ \hline
27 & 
$$ \{\{1, 8\}, \{1, 9\}, \{1, 10\}, \{1, 11\}, \{2, 5\}, \{2, 6\}, \{2, 7\}, \{2, 9\}, \{2, 12\},  \{3, 4\},$$ \\
& $$ \{3, 5\}, \{3, 7\}, \{3, 9\}, \{3, 13\}, \{4, 6\}, \{4, 7\}, \{4, 10\}, \{4, 12\},  \{5, 6\}, $$ \\
& $$\{5, 10\}, \{5, 14\}, \{6, 11\}, \{6, 13\}, \{7, 11\}, \{7, 14\}, \{8, 12\},  \{8, 13\}, \{8, 14\}\}$$ \\ \hline
28 &
$$ \{\{1, 8\}, \{1, 9\}, \{1, 10\}, \{1, 11\}, \{2, 8\}, \{2, 9\},  \{2, 12\}, \{2, 13\}, \{3, 8\}, \{3, 9\}, $$ \\
& $$\{3, 14\}, \{3, 15\}, \{4, 8\}, \{4, 10\},  \{4, 12\}, \{4, 14\}, \{5, 8\}, \{5, 10\}, \{5, 13\}, $$ \\
& $$\{5, 15\}, \{6, 8\}, \{6, 11\},  \{6, 12\}, \{6, 15\}, \{7, 8\}, \{7, 11\}, \{7, 13\}, \{7, 14\}\}$$ \\ \hline
29 &
$$\{\{1, 8\}, \{1, 9\}, \{1, 10\}, \{1, 11\}, \{2, 8\}, \{2, 9\}, \{2, 12\}, \{2, 13\},  \{3, 6\}, \{3, 9\}, $$ \\
& $$\{3, 14\}, \{3, 15\}, \{4, 8\}, \{4, 10\}, \{4, 12\}, \{4, 14\},  \{5, 6\}, \{5, 10\}, \{5, 13\}, $$ \\
& $$\{5, 15\}, \{6, 8\}, \{6, 11\}, \{6, 12\}, \{7, 8\},  \{7, 11\}, \{7, 13\}, \{7, 14\}, \{8, 15\}\}$$ \\ \hline
30 &
$$ \{\{1, 8\}, \{1, 9\}, \{1, 10\}, \{1, 11\},  \{2, 8\}, \{2, 9\}, \{2, 12\}, \{2, 13\}, \{3, 5\}, \{3, 6\},$$ \\
& $$\{3, 7\}, \{3, 9\}, \{3, 14\},   \{4, 7\}, \{4, 10\}, \{4, 12\}, \{4, 14\}, \{5, 10\}, \{5, 13\},$$ \\
& $$ \{5, 15\}, \{6, 11\},  \{6, 12\}, \{6, 15\}, \{7, 8\}, \{7, 11\}, \{7, 13\}, \{8, 14\}, \{8, 15\}\}$$ \\ \hline
31 &
$$ \{\{1, 8\}, \{1, 9\}, \{1, 10\}, \{1, 11\}, \{2, 7\}, \{2, 9\}, \{2, 12\}, \{2, 13\},  \{3, 4\}, \{3, 9\}, $$ \\
& $$\{3, 14\}, \{3, 15\}, \{4, 7\}, \{4, 8\}, \{4, 10\}, \{4, 12\},  \{5, 6\}, \{5, 10\}, \{5, 13\},$$ \\
& $$ \{5, 14\}, \{6, 11\}, \{6, 12\}, \{6, 15\}, \{7, 8\},  \{7, 11\}, \{7, 14\}, \{8, 13\}, \{8, 15\}\}$$ \\ \hline
32 &
$$ \{\{1, 8\}, \{1, 9\}, \{1, 10\}, \{1, 11\},  \{2, 7\}, \{2, 9\}, \{2, 12\}, \{2, 13\}, \{3, 4\}, \{3, 9\}, $$ \\
& $$\{3, 14\}, \{3, 15\},  \{4, 10\}, \{4, 12\}, \{4, 16\}, \{5, 6\}, \{5, 10\}, \{5, 13\}, \{5, 14\},$$ \\
& $$ \{6, 11\},  \{6, 12\}, \{6, 15\}, \{7, 11\}, \{7, 14\}, \{7, 16\}, \{8, 13\}, \{8, 15\}, \{8, 16\}\}$$
      \end{tabular}
\caption{Edges Lists for graphs in $K_8$ family (continued)}      
\label{tblK82}%
   \end{center}
\end{table}

\section{$K_8$ family graphs are not $3$-apex}
For each cousin $C_n$  ($5 \leq n \leq 32$) in the $K_8$ family, a $\yt$ move at vertex $9$ results in another graph in that family. 
We can show $C_n$ is not $3$-apex by arguing that $C_n-9$ is not $2$-apex. 
For this we list, in Table~\ref{tblN3A}, edges that are to be deleted or contracted in $C_n-9$ to produce a minor which
is a cousin in the $K_{3^2,1^2}$ family. Cousins in that family are numbered as in~\cite{GMN} 

For example, deleting edges $\{\{4,6\},\{4,8\},\{6,8\}\}$ of $C_5 - 9$ results in cousin 4 of the $K_{3^2,1^2}$ family.
In the case of $C_8-9$, first delete edges $\{4,6\}$ and $\{4,8\}$ and then contract edge $\{8,11\}$ to 
realize cousin 2 of the $K_{3^2,1^2}$ family.

\begin{table}[h]\begin{center}\begin{tabular}{c|l|c}

  $n$ & Edges deleted/contracted & $K_{3^2,1^2}$ cousin \\\hline\hline

  $5$  & $\{\{4,6\},\{4,8\},\{6,8\}\}$ & $4$ \\

  $6$  & $\{\{4,7\},\{4,8\},\{7,8\}\}$ & $4$ \\

  $7$  & $\{\{4,6\},\{4,8\},\{6,8\}\}$ &  $6$ \\

  $8$  & $\{\{4,6\},\{4,7\}\} \;/\;\{8,11\}$ & $2$ \\

  $9$  & $\{\{4,7\},\{4,8\},\{7,8\}\}$ & $8$ \\

  $10$ & $\{\{5,6\},\{5,8\},\{6,8\}\}$ & $9$ \\

  $11$ & $\{\{1,7\},\{1,8\},\{7,8\}\}$ & $11$ \\

  $12$ & $\{\{4,7\},\{4,8\},\{7,8\}\}$ & $9$ \\

  $13$ & $\{\{5,6\},\{5,7\},\{6,7\}\}$ &  $14$ \\

  $14$ & $\{\{5,7\},\{5,8\},\{7,8\}\}$ & $16$ \\

  $15$ & $\{\{4,7\},\{4,8\},\{7,8\}\}$ & $17$ \\

  $16$ & $\{\{4,7\},\{4,8\},\{7,8\}\}$ & $18$ \\

  $17$ & $\{\{4,7\},\{4,8\},\{7,8\}\}$ & $18$ \\

  $18$ & $\{\{4,7\},\{4,8\},\{7,8\}\}$ & $19$ \\

  $19$ & $\;/\;\{\{1,10\},\{2,12\},\{3,13\}\}$ & $3$ \\

  $20$ & $\{\{5,7\},\{5,7\},\{7,8\}\}$ & $23$ \\

  $21$ & $\{4,7\} \;/\;\{\{2,11\},\{8,13\}\}$ & $5$ \\

  $22$ & $\{\{5,6\},\{5,8\},\{6,8\}\}$ & $29$ \\

  $23$ & $\{\{5,6\},\{5,8\},\{6,8\}\}$ & $29$ \\

  $24$ & $\{\{5,6\},\{5,8\},\{6,8\}\}$ & $29$ \\

  $25$ & $\{\{4,7\},\{4,8\},\{7,8\}\}$ & $30$ \\

  $26$ & $\{5,7\} \;/\;\{\{2,13\},\{3,14\}\}$ & $10$ \\

  $27$ & $\;/\;\{\{1,10\},\{2,12\},\{7,14\}\}$ & $7$ \\

  $28$ & $\;/\;\{\{1,10\},\{2,12\},\{3,15\}\}$ & $13$ \\

  $29$ & $\;/\;\{\{1,10\},\{2,13\},\{3,14\}\}$ & $14$ \\

  $30$ & $\;/\;\{\{1,10\},\{3,14\},\{6,15\}\}$ & $10$ \\

  $31$ & $\{\{4,7\},\{4,8\},\{7,8\}\}$ & $42$ \\

  $32$ & $\;/\;\{\{1,10\},\{2,13\},\{7,16\}\}$ & $24$ \\

\hline\end{tabular}
\caption{Identifying a minor of $C_n-9$ in the $K_{3^2,1^2}$ family}
\label{tblN3A}%
\end{center}\end{table}

\section{Proper minors are $3$-apex}
%
For each cousin $C_n$ ($5 \leq n \leq 32$) in the $K_8$ family,
each proper minor of $C_n$ is $3$-apex.
We show this in Tables~\ref{tblPM3A1}--\ref{tblPM3A7} by considering 
each edge $e$ of $C_n$ (distinct up to symmetry) and for the graphs 
$C_n-e$ and $C_n/e$ (edge deletion and edge contraction, respectively) 
listing three vertices whose deletion gives a planar subgraph.

In Tables~\ref{tblPM3A1}--\ref{tblPM3A7} we use the convention that the vertex
that is the result of an edge contraction adopts the name of the 
smaller of the two vertex names on the edge that was contracted.
For example if we were to contract the edge $\{1,8\}$,
the new resulting vertex would be referred to as vertex $1$.

\begin{table}[hp]\begin{center}\begin{tabular}{c|c|l}
$n$ & Edge $e$ & Apex Sets  \\\hline\hline

  5 & $\{1, 8\}$ & $C_{5}-e$: 2, 3, 4; $C_{5}/e$: 1, 2, 3 \\
    & $\{1, 9\}$ & $C_{5}-e$: 2, 4, 5; $C_{5}/e$: 1, 4, 5 \\
    & $\{2, 4\}$ & $C_{5}-e$: 1, 3, 5; $C_{5}/e$: 1, 2, 3 \\
    & $\{2, 8\}$ & $C_{5}-e$: 1, 3, 4; $C_{5}/e$: 1, 2, 3 \\
    & $\{2, 9\}$ & $C_{5}-e$: 1, 4, 5; $C_{5}/e$: 2, 4, 5 \\
    \hline
   
  6 & $\{1, 6\}$ & $C_{6}-e$: 2, 3, 4; $C_{6}/e$: 1, 2, 3 \\
    & $\{1, 7\}$ & $C_{6}-e$: 2, 3, 4; $C_{6}/e$: 1, 2, 3 \\
    & $\{1, 9\}$ & $C_{6}-e$: 2, 4, 5; $C_{6}/e$: 1, 4, 5 \\
    & $\{3, 5\}$ & $C_{6}-e$: 1, 2, 4; $C_{6}/e$: 1, 2, 3 \\
    & $\{3, 7\}$ & $C_{6}-e$: 1, 2, 4; $C_{6}/e$: 1, 2, 3 \\
    & $\{3, 9\}$ & $C_{6}-e$: 1, 4, 5; $C_{6}/e$: 3, 4, 5 \\
    & $\{7, 8\}$ & $C_{6}-e$: 1, 2, 3; $C_{6}/e$: 1, 2, 3 \\
    \hline
   
  7 & $\{1, 6\}$ & $C_{7}-e$: 2, 3, 4; $C_{7}/e$: 1, 2, 3 \\
    & $\{1, 8\}$ & $C_{7}-e$: 2, 3, 4; $C_{7}/e$: 1, 2, 3 \\
    & $\{1, 9\}$ & $C_{7}-e$: 2, 4, 5; $C_{7}/e$: 1, 4, 5 \\
    & $\{1, 10\}$ & $C_{7}-e$: 2, 3, 4; $C_{7}/e$: 1, 2, 3 \\
    & $\{3, 4\}$ & $C_{7}-e$: 1, 2, 5; $C_{7}/e$: 1, 2, 3 \\
    & $\{3, 8\}$ & $C_{7}-e$: 1, 2, 4; $C_{7}/e$: 1, 2, 3 \\
    & $\{3, 9\}$ & $C_{7}-e$: 1, 4, 5; $C_{7}/e$: 3, 4, 5 \\
    & $\{4, 6\}$ & $C_{7}-e$: 1, 2, 3; $C_{7}/e$: 1, 2, 3 \\
    & $\{4, 8\}$ & $C_{7}-e$: 1, 2, 3; $C_{7}/e$: 1, 2, 3 \\
    & $\{4, 10\}$ & $C_{7}-e$: 1, 2, 3; $C_{7}/e$: 2, 3, 4 \\
    \hline
   
  8 & $\{1, 6\}$ & $C_{8}-e$: 2, 3, 4; $C_{8}/e$: 1, 2, 3 \\
    & $\{1, 9\}$ & $C_{8}-e$: 2, 4, 6; $C_{8}/e$: 1, 4, 6 \\
    & $\{2, 4\}$ & $C_{8}-e$: 1, 3, 6; $C_{8}/e$: 1, 2, 3 \\
    & $\{2, 6\}$ & $C_{8}-e$: 1, 3, 4; $C_{8}/e$: 1, 2, 3 \\
    & $\{2, 9\}$ & $C_{8}-e$: 1, 4, 6; $C_{8}/e$: 2, 4, 6 \\
    & $\{6, 11\}$ & $C_{8}-e$: 1, 2, 3; $C_{8}/e$: 1, 2, 3 \\
    \hline
   
  9 & $\{1, 8\}$ & $C_{9}-e$: 2, 3, 4; $C_{9}/e$: 1, 2, 3 \\
    & $\{1, 9\}$ & $C_{9}-e$: 2, 4, 5; $C_{9}/e$: 1, 4, 5 \\
    & $\{2, 5\}$ & $C_{9}-e$: 1, 3, 4; $C_{9}/e$: 1, 2, 3 \\
    & $\{2, 8\}$ & $C_{9}-e$: 1, 3, 4; $C_{9}/e$: 1, 2, 3 \\
    & $\{2, 9\}$ & $C_{9}-e$: 1, 4, 5; $C_{9}/e$: 2, 4, 5 \\
    & $\{2, 12\}$ & $C_{9}-e$: 1, 3, 4; $C_{9}/e$: 1, 2, 3 \\
    & $\{3, 5\}$ & $C_{9}-e$: 1, 2, 4; $C_{9}/e$: 1, 2, 3 \\
    & $\{3, 8\}$ & $C_{9}-e$: 1, 2, 4; $C_{9}/e$: 1, 2, 3 \\
    & $\{3, 9\}$ & $C_{9}-e$: 1, 4, 5; $C_{9}/e$: 3, 4, 5 \\
   
\hline\end{tabular}
\caption{Every proper minor of $C_n$ is $3$-apex.}
\label{tblPM3A1}%
\end{center}\end{table}

\begin{table}[hp]\begin{center}\begin{tabular}{c|c|l}
$n$ & Edge $e$ & Apex Sets  \\\hline\hline
   
  10 & $\{1, 8\}$ & $C_{10}-e$: 2, 3, 4; $C_{10}/e$: 1, 2, 3 \\
    & $\{1, 9\}$ & $C_{10}-e$: 2, 4, 5; $C_{10}/e$: 1, 4, 5 \\
    & $\{1, 11\}$ & $C_{10}-e$: 2, 3, 4; $C_{10}/e$: 1, 2, 3 \\
    & $\{2, 5\}$ & $C_{10}-e$: 1, 3, 4; $C_{10}/e$: 1, 2, 3 \\
    & $\{2, 6\}$ & $C_{10}-e$: 1, 3, 4; $C_{10}/e$: 1, 2, 3 \\
    & $\{2, 9\}$ & $C_{10}-e$: 1, 4, 5; $C_{10}/e$: 2, 4, 5 \\
    & $\{2, 12\}$ & $C_{10}-e$: 1, 3, 4; $C_{10}/e$: 1, 2, 3 \\
    & $\{3, 5\}$ & $C_{10}-e$: 1, 2, 4; $C_{10}/e$: 1, 2, 3 \\
    & $\{3, 6\}$ & $C_{10}-e$: 1, 2, 4; $C_{10}/e$: 1, 2, 3 \\
    & $\{3, 8\}$ & $C_{10}-e$: 1, 2, 4; $C_{10}/e$: 1, 2, 3 \\
    & $\{3, 9\}$ & $C_{10}-e$: 1, 4, 5; $C_{10}/e$: 3, 4, 5 \\
    & $\{6, 8\}$ & $C_{10}-e$: 1, 2, 3; $C_{10}/e$: 1, 2, 3 \\
    & $\{6, 11\}$ & $C_{10}-e$: 1, 2, 3; $C_{10}/e$: 2, 3, 4 \\
    & $\{8, 12\}$ & $C_{10}-e$: 1, 2, 3; $C_{10}/e$: 1, 3, 5 \\
    \hline
   
  11 & $\{1, 6\}$ & $C_{11}-e$: 2, 3, 4; $C_{11}/e$: 1, 2, 3 \\
    & $\{1, 7\}$ & $C_{11}-e$: 2, 3, 4; $C_{11}/e$: 1, 2, 3 \\
    & $\{1, 9\}$ & $C_{11}-e$: 2, 4, 5; $C_{11}/e$: 1, 4, 5 \\
    & $\{7, 8\}$ & $C_{11}-e$: 1, 2, 3; $C_{11}/e$: 1, 2, 3 \\
    \hline
   
  12 & $\{1, 6\}$ & $C_{12}-e$: 2, 3, 4; $C_{12}/e$: 1, 2, 3 \\
    & $\{1, 8\}$ & $C_{12}-e$: 2, 3, 4; $C_{12}/e$: 1, 2, 3 \\
    & $\{1, 9\}$ & $C_{12}-e$: 2, 4, 5; $C_{12}/e$: 1, 4, 5 \\
    & $\{2, 5\}$ & $C_{12}-e$: 1, 3, 4; $C_{12}/e$: 1, 2, 3 \\
    & $\{2, 7\}$ & $C_{12}-e$: 1, 3, 4; $C_{12}/e$: 1, 2, 3 \\
    & $\{2, 8\}$ & $C_{12}-e$: 1, 3, 4; $C_{12}/e$: 1, 2, 3 \\
    & $\{2, 9\}$ & $C_{12}-e$: 1, 4, 5; $C_{12}/e$: 2, 4, 5 \\
    & $\{2, 11\}$ & $C_{12}-e$: 1, 3, 4; $C_{12}/e$: 1, 2, 3 \\
    & $\{6, 7\}$ & $C_{12}-e$: 1, 2, 3; $C_{12}/e$: 1, 2, 3 \\
    & $\{6, 8\}$ & $C_{12}-e$: 1, 2, 3; $C_{12}/e$: 1, 2, 3 \\
    & $\{6, 11\}$ & $C_{12}-e$: 1, 2, 3; $C_{12}/e$: 1, 3, 5 \\
    \hline
   
  13 & $\{1, 6\}$ & $C_{13}-e$: 2, 3, 4; $C_{13}/e$: 1, 2, 3 \\
    & $\{1, 7\}$ & $C_{13}-e$: 2, 3, 4; $C_{13}/e$: 1, 2, 3 \\
    & $\{1, 9\}$ & $C_{13}-e$: 2, 4, 5; $C_{13}/e$: 1, 4, 5 \\
    & $\{1, 10\}$ & $C_{13}-e$: 2, 3, 4; $C_{13}/e$: 1, 2, 3 \\
    & $\{3, 4\}$ & $C_{13}-e$: 1, 2, 5; $C_{13}/e$: 1, 2, 3 \\
    & $\{3, 5\}$ & $C_{13}-e$: 1, 2, 4; $C_{13}/e$: 1, 2, 3 \\
    & $\{3, 9\}$ & $C_{13}-e$: 1, 4, 5; $C_{13}/e$: 3, 4, 5 \\
    & $\{3, 12\}$ & $C_{13}-e$: 1, 2, 4; $C_{13}/e$: 1, 2, 3 \\
    & $\{4, 7\}$ & $C_{13}-e$: 1, 2, 3; $C_{13}/e$: 1, 2, 3 \\
    & $\{4, 10\}$ & $C_{13}-e$: 1, 2, 3; $C_{13}/e$: 2, 3, 4 \\
    & $\{5, 6\}$ & $C_{13}-e$: 1, 2, 3; $C_{13}/e$: 1, 2, 3 \\
    & $\{5, 7\}$ & $C_{13}-e$: 1, 2, 3; $C_{13}/e$: 1, 2, 3 \\
    & $\{5, 10\}$ & $C_{13}-e$: 1, 2, 3; $C_{13}/e$: 2, 3, 5 \\
    & $\{7, 12\}$ & $C_{13}-e$: 1, 2, 3; $C_{13}/e$: 1, 2, 4 \\
    \hline
   
\hline\end{tabular}
\caption{Every proper minor of $C_n$ is $3$-apex (continued)}
\label{tblPM3A2}%
\end{center}\end{table}

\begin{table}[hp]\begin{center}\begin{tabular}{c|c|l}
$n$ & Edge $e$ & Apex Sets  \\\hline\hline
   
  14 & $\{1, 6\}$ & $C_{14}-e$: 2, 3, 4; $C_{14}/e$: 1, 2, 3 \\
    & $\{1, 7\}$ & $C_{14}-e$: 2, 3, 4; $C_{14}/e$: 1, 2, 3 \\
    & $\{1, 9\}$ & $C_{14}-e$: 2, 4, 5; $C_{14}/e$: 1, 4, 5 \\
    & $\{3, 5\}$ & $C_{14}-e$: 1, 2, 4; $C_{14}/e$: 1, 2, 3 \\
    & $\{3, 8\}$ & $C_{14}-e$: 1, 2, 4; $C_{14}/e$: 1, 2, 3 \\
    & $\{3, 9\}$ & $C_{14}-e$: 1, 4, 5; $C_{14}/e$: 3, 4, 5 \\
    \hline
   
  15 & $\{1, 8\}$ & $C_{15}-e$: 2, 3, 4; $C_{15}/e$: 1, 2, 3 \\
    & $\{1, 9\}$ & $C_{15}-e$: 2, 4, 5; $C_{15}/e$: 1, 4, 5 \\
    & $\{1, 10\}$ & $C_{15}-e$: 2, 3, 4; $C_{15}/e$: 1, 2, 3 \\
    & $\{3, 4\}$ & $C_{15}-e$: 1, 2, 5; $C_{15}/e$: 1, 2, 3 \\
    & $\{3, 8\}$ & $C_{15}-e$: 1, 2, 4; $C_{15}/e$: 1, 2, 3 \\
    & $\{3, 9\}$ & $C_{15}-e$: 1, 4, 5; $C_{15}/e$: 3, 4, 5 \\
    & $\{4, 7\}$ & $C_{15}-e$: 1, 2, 3; $C_{15}/e$: 1, 2, 3 \\
    & $\{4, 8\}$ & $C_{15}-e$: 1, 2, 3; $C_{15}/e$: 1, 2, 3 \\
    & $\{4, 10\}$ & $C_{15}-e$: 1, 2, 3; $C_{15}/e$: 2, 3, 4 \\
    \hline
   
  16 & $\{1, 8\}$ & $C_{16}-e$: 2, 3, 4; $C_{16}/e$: 1, 2, 3 \\
    & $\{1, 9\}$ & $C_{16}-e$: 2, 4, 5; $C_{16}/e$: 1, 4, 5 \\
    & $\{1, 10\}$ & $C_{16}-e$: 2, 3, 4; $C_{16}/e$: 1, 2, 3 \\
    & $\{1, 11\}$ & $C_{16}-e$: 2, 3, 4; $C_{16}/e$: 1, 2, 3 \\
    & $\{3, 4\}$ & $C_{16}-e$: 1, 2, 5; $C_{16}/e$: 1, 2, 3 \\
    & $\{3, 5\}$ & $C_{16}-e$: 1, 2, 4; $C_{16}/e$: 1, 2, 3 \\
    & $\{3, 7\}$ & $C_{16}-e$: 1, 2, 4; $C_{16}/e$: 1, 2, 3 \\
    & $\{3, 9\}$ & $C_{16}-e$: 1, 4, 5; $C_{16}/e$: 3, 4, 5 \\
    & $\{4, 7\}$ & $C_{16}-e$: 1, 2, 3; $C_{16}/e$: 1, 2, 3 \\
    & $\{4, 10\}$ & $C_{16}-e$: 1, 2, 3; $C_{16}/e$: 2, 3, 4 \\
    & $\{5, 6\}$ & $C_{16}-e$: 1, 2, 3; $C_{16}/e$: 1, 2, 3 \\
    & $\{5, 7\}$ & $C_{16}-e$: 1, 2, 3; $C_{16}/e$: 1, 2, 3 \\
    & $\{5, 10\}$ & $C_{16}-e$: 1, 2, 3; $C_{16}/e$: 2, 3, 5 \\
    & $\{5, 13\}$ & $C_{16}-e$: 1, 2, 3; $C_{16}/e$: 1, 3, 4 \\
    & $\{7, 8\}$ & $C_{16}-e$: 1, 2, 3; $C_{16}/e$: 1, 2, 3 \\
    & $\{7, 11\}$ & $C_{16}-e$: 1, 2, 3; $C_{16}/e$: 2, 3, 4 \\
    \hline
   
  17 & $\{1, 8\}$ & $C_{17}-e$: 2, 3, 4; $C_{17}/e$: 1, 2, 3 \\
    & $\{1, 9\}$ & $C_{17}-e$: 2, 4, 5; $C_{17}/e$: 1, 4, 5 \\
    & $\{2, 5\}$ & $C_{17}-e$: 1, 3, 4; $C_{17}/e$: 1, 2, 3 \\
    & $\{2, 8\}$ & $C_{17}-e$: 1, 3, 4; $C_{17}/e$: 1, 2, 3 \\
    & $\{2, 9\}$ & $C_{17}-e$: 1, 4, 5; $C_{17}/e$: 2, 4, 5 \\
    & $\{2, 12\}$ & $C_{17}-e$: 1, 3, 4; $C_{17}/e$: 1, 2, 3 \\
    \hline
   
\hline\end{tabular}
\caption{Every proper minor of $C_n$ is $3$-apex (continued)}
\label{tblPM3A3}%
\end{center}\end{table}

\begin{table}[hp]\begin{center}\begin{tabular}{c|c|l}
$n$ & Edge $e$ & Apex Sets  \\\hline\hline
   
  18 & $\{1, 8\}$ & $C_{18}-e$: 2, 3, 4; $C_{18}/e$: 1, 2, 3 \\
    & $\{1, 9\}$ & $C_{18}-e$: 2, 4, 5; $C_{18}/e$: 1, 4, 5 \\
    & $\{1, 11\}$ & $C_{18}-e$: 2, 3, 4; $C_{18}/e$: 1, 2, 3 \\
    & $\{2, 5\}$ & $C_{18}-e$: 1, 3, 4; $C_{18}/e$: 1, 2, 3 \\
    & $\{2, 7\}$ & $C_{18}-e$: 1, 3, 4; $C_{18}/e$: 1, 2, 3 \\
    & $\{2, 8\}$ & $C_{18}-e$: 1, 3, 4; $C_{18}/e$: 1, 2, 3 \\
    & $\{2, 9\}$ & $C_{18}-e$: 1, 4, 5; $C_{18}/e$: 2, 4, 5 \\
    & $\{2, 12\}$ & $C_{18}-e$: 1, 3, 4; $C_{18}/e$: 1, 2, 3 \\
    & $\{3, 6\}$ & $C_{18}-e$: 1, 2, 4; $C_{18}/e$: 1, 2, 3 \\
    & $\{3, 7\}$ & $C_{18}-e$: 1, 2, 4; $C_{18}/e$: 1, 2, 3 \\
    & $\{3, 9\}$ & $C_{18}-e$: 1, 4, 5; $C_{18}/e$: 3, 4, 5 \\
    & $\{3, 13\}$ & $C_{18}-e$: 1, 2, 4; $C_{18}/e$: 1, 2, 3 \\
    & $\{6, 8\}$ & $C_{18}-e$: 1, 2, 3; $C_{18}/e$: 1, 2, 3 \\
    & $\{6, 11\}$ & $C_{18}-e$: 1, 2, 3; $C_{18}/e$: 2, 3, 4 \\
    & $\{6, 12\}$ & $C_{18}-e$: 1, 2, 3; $C_{18}/e$: 1, 3, 5 \\
    & $\{7, 8\}$ & $C_{18}-e$: 1, 2, 3; $C_{18}/e$: 1, 2, 3 \\
    & $\{7, 11\}$ & $C_{18}-e$: 1, 2, 3; $C_{18}/e$: 2, 3, 4 \\
    & $\{8, 13\}$ & $C_{18}-e$: 1, 2, 3; $C_{18}/e$: 1, 2, 4 \\
    \hline
   
  19 & $\{1, 8\}$ & $C_{19}-e$: 2, 3, 4; $C_{19}/e$: 1, 2, 3 \\
    & $\{1, 9\}$ & $C_{19}-e$: 2, 4, 5; $C_{19}/e$: 1, 4, 5 \\
    & $\{1, 11\}$ & $C_{19}-e$: 2, 3, 4; $C_{19}/e$: 1, 2, 3 \\
    & $\{2, 5\}$ & $C_{19}-e$: 1, 3, 4; $C_{19}/e$: 1, 2, 3 \\
    & $\{2, 6\}$ & $C_{19}-e$: 1, 3, 4; $C_{19}/e$: 1, 2, 3 \\
    & $\{2, 9\}$ & $C_{19}-e$: 1, 4, 5; $C_{19}/e$: 2, 4, 5 \\
    & $\{2, 12\}$ & $C_{19}-e$: 1, 3, 4; $C_{19}/e$: 1, 2, 3 \\
    & $\{6, 8\}$ & $C_{19}-e$: 1, 2, 3; $C_{19}/e$: 1, 2, 3 \\
    & $\{6, 11\}$ & $C_{19}-e$: 1, 2, 3; $C_{19}/e$: 2, 3, 4 \\
    & $\{8, 12\}$ & $C_{19}-e$: 1, 2, 3; $C_{19}/e$: 1, 3, 5 \\
    \hline
   
  20 & $\{1, 8\}$ & $C_{20}-e$: 2, 3, 4; $C_{20}/e$: 1, 2, 3 \\
    & $\{1, 9\}$ & $C_{20}-e$: 2, 4, 5; $C_{20}/e$: 1, 4, 5 \\
    & $\{1, 10\}$ & $C_{20}-e$: 2, 3, 4; $C_{20}/e$: 1, 2, 3 \\
    & $\{2, 5\}$ & $C_{20}-e$: 1, 3, 4; $C_{20}/e$: 1, 2, 3 \\
    & $\{2, 6\}$ & $C_{20}-e$: 1, 3, 4; $C_{20}/e$: 1, 2, 3 \\
    & $\{2, 7\}$ & $C_{20}-e$: 1, 3, 4; $C_{20}/e$: 1, 2, 3 \\
    & $\{2, 9\}$ & $C_{20}-e$: 1, 4, 5; $C_{20}/e$: 2, 4, 5 \\
    & $\{2, 12\}$ & $C_{20}-e$: 1, 3, 4; $C_{20}/e$: 1, 2, 3 \\
    & $\{4, 6\}$ & $C_{20}-e$: 1, 2, 3; $C_{20}/e$: 1, 2, 3 \\
    & $\{4, 7\}$ & $C_{20}-e$: 1, 2, 3; $C_{20}/e$: 1, 2, 3 \\
    & $\{4, 10\}$ & $C_{20}-e$: 1, 2, 3; $C_{20}/e$: 2, 3, 4 \\
    & $\{4, 12\}$ & $C_{20}-e$: 1, 2, 3; $C_{20}/e$: 1, 3, 4 \\
    & $\{5, 7\}$ & $C_{20}-e$: 1, 2, 3; $C_{20}/e$: 1, 2, 3 \\
    & $\{5, 8\}$ & $C_{20}-e$: 1, 2, 3; $C_{20}/e$: 1, 2, 3 \\
    & $\{5, 10\}$ & $C_{20}-e$: 1, 2, 3; $C_{20}/e$: 2, 3, 5 \\
    & $\{8, 12\}$ & $C_{20}-e$: 1, 2, 3; $C_{20}/e$: 1, 3, 5 \\
    \hline
   
\hline\end{tabular}
\caption{Every proper minor of $C_n$ is $3$-apex (continued)}
\label{tblPM3A4}%
\end{center}\end{table}

\begin{table}[hp]\begin{center}\begin{tabular}{c|c|l}
$n$ & Edge $e$ & Apex Sets  \\\hline\hline
   
  21 & $\{1, 6\}$ & $C_{21}-e$: 2, 3, 4; $C_{21}/e$: 1, 2, 3 \\
    & $\{1, 8\}$ & $C_{21}-e$: 2, 3, 4; $C_{21}/e$: 1, 2, 3 \\
    & $\{1, 9\}$ & $C_{21}-e$: 2, 4, 6; $C_{21}/e$: 1, 4, 6 \\
    & $\{2, 5\}$ & $C_{21}-e$: 1, 3, 6; $C_{21}/e$: 1, 2, 3 \\
    & $\{2, 7\}$ & $C_{21}-e$: 1, 3, 4; $C_{21}/e$: 1, 2, 3 \\
    & $\{2, 8\}$ & $C_{21}-e$: 1, 3, 4; $C_{21}/e$: 1, 2, 3 \\
    & $\{2, 9\}$ & $C_{21}-e$: 1, 4, 6; $C_{21}/e$: 2, 4, 6 \\
    & $\{2, 11\}$ & $C_{21}-e$: 1, 3, 4; $C_{21}/e$: 1, 2, 3 \\
    & $\{6, 11\}$ & $C_{21}-e$: 1, 2, 3; $C_{21}/e$: 1, 3, 6 \\
    & $\{6, 13\}$ & $C_{21}-e$: 1, 2, 3; $C_{21}/e$: 1, 2, 3 \\
    & $\{8, 13\}$ & $C_{21}-e$: 1, 2, 3; $C_{21}/e$: 1, 2, 3 \\
    \hline
   
  22 & $\{1, 8\}$ & $C_{22}-e$: 2, 3, 4; $C_{22}/e$: 1, 2, 3 \\
    & $\{1, 9\}$ & $C_{22}-e$: 2, 4, 5; $C_{22}/e$: 1, 4, 5 \\
    & $\{3, 5\}$ & $C_{22}-e$: 1, 2, 4; $C_{22}/e$: 1, 2, 3 \\
    & $\{3, 8\}$ & $C_{22}-e$: 1, 2, 4; $C_{22}/e$: 1, 2, 3 \\
    & $\{3, 9\}$ & $C_{22}-e$: 1, 4, 5; $C_{22}/e$: 3, 4, 5 \\
    \hline
   
  23 & $\{1, 8\}$ & $C_{23}-e$: 2, 3, 4; $C_{23}/e$: 1, 2, 3 \\
    & $\{1, 9\}$ & $C_{23}-e$: 2, 4, 5; $C_{23}/e$: 1, 4, 5 \\
    & $\{1, 10\}$ & $C_{23}-e$: 2, 3, 4; $C_{23}/e$: 1, 2, 3 \\
    & $\{1, 11\}$ & $C_{23}-e$: 2, 3, 4; $C_{23}/e$: 1, 2, 3 \\
    & $\{3, 5\}$ & $C_{23}-e$: 1, 2, 4; $C_{23}/e$: 1, 2, 3 \\
    & $\{3, 7\}$ & $C_{23}-e$: 1, 2, 4; $C_{23}/e$: 1, 2, 3 \\
    & $\{3, 9\}$ & $C_{23}-e$: 1, 4, 5; $C_{23}/e$: 3, 4, 5 \\
    & $\{3, 14\}$ & $C_{23}-e$: 1, 2, 4; $C_{23}/e$: 1, 2, 3 \\
    & $\{4, 7\}$ & $C_{23}-e$: 1, 2, 3; $C_{23}/e$: 1, 2, 3 \\
    & $\{4, 10\}$ & $C_{23}-e$: 1, 2, 3; $C_{23}/e$: 2, 3, 4 \\
    & $\{4, 14\}$ & $C_{23}-e$: 1, 2, 3; $C_{23}/e$: 1, 2, 4 \\
    & $\{5, 6\}$ & $C_{23}-e$: 1, 2, 3; $C_{23}/e$: 1, 2, 3 \\
    & $\{5, 8\}$ & $C_{23}-e$: 1, 2, 3; $C_{23}/e$: 1, 2, 3 \\
    & $\{5, 10\}$ & $C_{23}-e$: 1, 2, 3; $C_{23}/e$: 2, 3, 5 \\
    & $\{5, 13\}$ & $C_{23}-e$: 1, 2, 3; $C_{23}/e$: 1, 3, 4 \\
    & $\{7, 8\}$ & $C_{23}-e$: 1, 2, 3; $C_{23}/e$: 1, 2, 3 \\
    & $\{7, 11\}$ & $C_{23}-e$: 1, 2, 3; $C_{23}/e$: 2, 3, 4 \\
    & $\{8, 14\}$ & $C_{23}-e$: 1, 2, 3; $C_{23}/e$: 1, 2, 5 \\
    \hline
   
  24 & $\{1, 8\}$ & $C_{24}-e$: 2, 3, 4; $C_{24}/e$: 1, 2, 3 \\
    & $\{1, 9\}$ & $C_{24}-e$: 2, 4, 5; $C_{24}/e$: 1, 4, 5 \\
    & $\{1, 10\}$ & $C_{24}-e$: 2, 3, 4; $C_{24}/e$: 1, 2, 3 \\
    & $\{3, 5\}$ & $C_{24}-e$: 1, 2, 4; $C_{24}/e$: 1, 2, 3 \\
    & $\{3, 8\}$ & $C_{24}-e$: 1, 2, 4; $C_{24}/e$: 1, 2, 3 \\
    & $\{3, 9\}$ & $C_{24}-e$: 1, 4, 5; $C_{24}/e$: 3, 4, 5 \\
    & $\{5, 6\}$ & $C_{24}-e$: 1, 2, 4; $C_{24}/e$: 1, 2, 3 \\
    & $\{5, 10\}$ & $C_{24}-e$: 1, 2, 3; $C_{24}/e$: 2, 3, 5 \\
    \hline
   
\hline\end{tabular}
\caption{Every proper minor of $C_n$ is $3$-apex (continued)}
\label{tblPM3A5}%
\end{center}\end{table}

\begin{table}[hp]\begin{center}\begin{tabular}{c|c|l}
$n$ & Edge $e$ & Apex Sets  \\\hline\hline
   
  25 & $\{1, 8\}$ & $C_{25}-e$: 2, 3, 4; $C_{25}/e$: 1, 2, 3 \\
    & $\{1, 9\}$ & $C_{25}-e$: 2, 4, 5; $C_{25}/e$: 1, 4, 5 \\
    & $\{1, 11\}$ & $C_{25}-e$: 2, 3, 4; $C_{25}/e$: 1, 2, 3 \\
    & $\{2, 7\}$ & $C_{25}-e$: 1, 3, 4; $C_{25}/e$: 1, 2, 3 \\
    & $\{2, 9\}$ & $C_{25}-e$: 1, 4, 5; $C_{25}/e$: 2, 4, 5 \\
    & $\{2, 12\}$ & $C_{25}-e$: 1, 3, 4; $C_{25}/e$: 1, 2, 3 \\
    & $\{2, 13\}$ & $C_{25}-e$: 1, 3, 4; $C_{25}/e$: 1, 2, 3 \\
    & $\{3, 4\}$ & $C_{25}-e$: 1, 2, 5; $C_{25}/e$: 1, 2, 3 \\
    & $\{3, 6\}$ & $C_{25}-e$: 1, 2, 4; $C_{25}/e$: 1, 2, 3 \\
    & $\{3, 8\}$ & $C_{25}-e$: 1, 2, 4; $C_{25}/e$: 1, 2, 3 \\
    & $\{3, 9\}$ & $C_{25}-e$: 1, 4, 5; $C_{25}/e$: 3, 4, 5 \\
    & $\{3, 14\}$ & $C_{25}-e$: 1, 2, 4; $C_{25}/e$: 1, 2, 3 \\
    & $\{6, 8\}$ & $C_{25}-e$: 1, 2, 3; $C_{25}/e$: 1, 2, 3 \\
    & $\{6, 11\}$ & $C_{25}-e$: 1, 2, 3; $C_{25}/e$: 2, 3, 4 \\
    & $\{6, 12\}$ & $C_{25}-e$: 1, 2, 3; $C_{25}/e$: 1, 3, 5 \\
    & $\{8, 13\}$ & $C_{25}-e$: 1, 2, 3; $C_{25}/e$: 1, 3, 4 \\
    \hline
   
  26 & $\{1, 8\}$ & $C_{26}-e$: 2, 3, 4; $C_{26}/e$: 1, 2, 3 \\
    & $\{1, 9\}$ & $C_{26}-e$: 2, 4, 5; $C_{26}/e$: 1, 4, 5 \\
    & $\{1, 10\}$ & $C_{26}-e$: 2, 3, 4; $C_{26}/e$: 1, 2, 3 \\
    & $\{1, 11\}$ & $C_{26}-e$: 2, 3, 4; $C_{26}/e$: 1, 2, 3 \\
    & $\{3, 4\}$ & $C_{26}-e$: 1, 2, 5; $C_{26}/e$: 1, 2, 3 \\
    & $\{3, 5\}$ & $C_{26}-e$: 1, 2, 4; $C_{26}/e$: 1, 2, 3 \\
    & $\{3, 9\}$ & $C_{26}-e$: 1, 4, 5; $C_{26}/e$: 3, 4, 5 \\
    & $\{3, 14\}$ & $C_{26}-e$: 1, 2, 4; $C_{26}/e$: 1, 2, 3 \\
    & $\{4, 7\}$ & $C_{26}-e$: 1, 2, 3; $C_{26}/e$: 1, 2, 3 \\
    & $\{4, 10\}$ & $C_{26}-e$: 1, 2, 3; $C_{26}/e$: 2, 3, 4 \\
    & $\{5, 6\}$ & $C_{26}-e$: 1, 2, 3; $C_{26}/e$: 1, 2, 3 \\
    & $\{5, 7\}$ & $C_{26}-e$: 1, 2, 3; $C_{26}/e$: 1, 2, 3 \\
    & $\{5, 10\}$ & $C_{26}-e$: 1, 2, 3; $C_{26}/e$: 2, 3, 5 \\
    & $\{5, 13\}$ & $C_{26}-e$: 1, 2, 3; $C_{26}/e$: 1, 3, 4 \\
    & $\{7, 11\}$ & $C_{26}-e$: 1, 2, 3; $C_{26}/e$: 2, 3, 4 \\
    & $\{7, 14\}$ & $C_{26}-e$: 1, 2, 3; $C_{26}/e$: 1, 2, 4 \\
    \hline
   
  27 & $\{1, 8\}$ & $C_{27}-e$: 2, 3, 4; $C_{27}/e$: 1, 2, 3 \\
    & $\{1, 9\}$ & $C_{27}-e$: 2, 4, 5; $C_{27}/e$: 1, 4, 5 \\
    & $\{2, 5\}$ & $C_{27}-e$: 1, 3, 4; $C_{27}/e$: 1, 2, 3 \\
    & $\{2, 7\}$ & $C_{27}-e$: 1, 3, 4; $C_{27}/e$: 1, 2, 3 \\
    & $\{2, 9\}$ & $C_{27}-e$: 1, 4, 5; $C_{27}/e$: 2, 4, 5 \\
    \hline
   
  28 & $\{1, 8\}$ & $C_{28}-e$: 2, 3, 4; $C_{28}/e$: 1, 2, 3 \\
    & $\{1, 9\}$ & $C_{28}-e$: 2, 4, 5; $C_{28}/e$: 1, 4, 5 \\
    \hline
   
\hline\end{tabular}
\caption{Every proper minor of $C_n$ is $3$-apex. (continued)}
\label{tblPM3A6}%
\end{center}\end{table}

\begin{table}[hp]\begin{center}\begin{tabular}{c|c|l}
$n$ & Edge $e$ & Apex Sets  \\\hline\hline
   
  29 & $\{1, 8\}$ & $C_{29}-e$: 2, 3, 4; $C_{29}/e$: 1, 2, 3 \\
    & $\{1, 9\}$ & $C_{29}-e$: 2, 4, 5; $C_{29}/e$: 1, 4, 5 \\
    & $\{1, 11\}$ & $C_{29}-e$: 2, 3, 4; $C_{29}/e$: 1, 2, 3 \\
    & $\{3, 6\}$ & $C_{29}-e$: 1, 2, 4; $C_{29}/e$: 1, 2, 3 \\
    & $\{3, 9\}$ & $C_{29}-e$: 1, 4, 5; $C_{29}/e$: 3, 4, 5 \\
    & $\{3, 15\}$ & $C_{29}-e$: 1, 2, 4; $C_{29}/e$: 1, 2, 3 \\
    & $\{6, 8\}$ & $C_{29}-e$: 1, 2, 3; $C_{29}/e$: 1, 2, 3 \\
    & $\{6, 11\}$ & $C_{29}-e$: 1, 2, 3; $C_{29}/e$: 2, 3, 4 \\
    & $\{8, 15\}$ & $C_{29}-e$: 1, 2, 3; $C_{29}/e$: 1, 2, 4 \\
    \hline
   
  30 & $\{1, 8\}$ & $C_{30}-e$: 2, 3, 4; $C_{30}/e$: 1, 2, 3 \\
    & $\{1, 9\}$ & $C_{30}-e$: 2, 4, 5; $C_{30}/e$: 1, 4, 5 \\
    & $\{1, 10\}$ & $C_{30}-e$: 2, 3, 4; $C_{30}/e$: 1, 2, 3 \\
    & $\{1, 11\}$ & $C_{30}-e$: 2, 3, 4; $C_{30}/e$: 1, 2, 3 \\
    & $\{3, 7\}$ & $C_{30}-e$: 1, 2, 5; $C_{30}/e$: 1, 2, 3 \\
    & $\{3, 9\}$ & $C_{30}-e$: 1, 4, 5; $C_{30}/e$: 3, 4, 5 \\
    & $\{3, 14\}$ & $C_{30}-e$: 1, 2, 4; $C_{30}/e$: 1, 2, 3 \\
    & $\{4, 7\}$ & $C_{30}-e$: 1, 2, 5; $C_{30}/e$: 1, 2, 3 \\
    & $\{4, 10\}$ & $C_{30}-e$: 1, 2, 3; $C_{30}/e$: 2, 3, 4 \\
    & $\{4, 14\}$ & $C_{30}-e$: 1, 2, 3; $C_{30}/e$: 1, 2, 4 \\
    & $\{7, 11\}$ & $C_{30}-e$: 1, 2, 3; $C_{30}/e$: 2, 3, 4 \\
    \hline
   
  31 & $\{1, 8\}$ & $C_{31}-e$: 2, 3, 4; $C_{31}/e$: 1, 2, 3 \\
    & $\{1, 9\}$ & $C_{31}-e$: 2, 4, 5; $C_{31}/e$: 1, 4, 5 \\
    & $\{1, 10\}$ & $C_{31}-e$: 2, 3, 4; $C_{31}/e$: 1, 2, 3 \\
    & $\{4, 7\}$ & $C_{31}-e$: 1, 2, 3; $C_{31}/e$: 1, 2, 3 \\
    & $\{4, 10\}$ & $C_{31}-e$: 1, 2, 3; $C_{31}/e$: 2, 3, 4 \\
    & $\{5, 6\}$ & $C_{31}-e$: 1, 2, 3; $C_{31}/e$: 1, 2, 3 \\
    & $\{5, 10\}$ & $C_{31}-e$: 1, 2, 3; $C_{31}/e$: 2, 3, 5 \\
    \hline
   
  32 & $\{1, 8\}$ & $C_{32}-e$: 2, 3, 4; $C_{32}/e$: 1, 2, 3 \\
    & $\{1, 9\}$ & $C_{32}-e$: 2, 4, 5; $C_{32}/e$: 1, 4, 5 \\
   
\hline\end{tabular}
\caption{Every proper minor of $C_n$ is $3$-apex (continued)}
\label{tblPM3A7}%
\end{center}\end{table}

\end{document}